\documentclass[11pt,a4paper,reqno]{amsart}
\usepackage[utf8]{inputenc}
\usepackage{amsmath}
\usepackage{amsfonts}
\usepackage{titlesec}
\usepackage[capposition=top]{floatrow}
\usepackage[margin=1.2in]{geometry} 
\usepackage{changepage}
\usepackage{amssymb}
\usepackage{relsize}
\usepackage{xcolor}
\usepackage{marginnote}
\usepackage{mathtools}
\usepackage{amsthm}
\usepackage{appendix}
\usepackage{bbm}
\usepackage{shadethm}
\usepackage{expl3}
\usepackage{graphicx}
\usepackage{fancyhdr}
\usepackage{setspace}
\usepackage{braket}
\usepackage{epstopdf}
\usepackage{titlesec}
\let\sssum\sum

\let\bigcuppp\bigcup
\let\bigsqcuppp\bigsqcup
\usepackage{fouriernc} 
\renewcommand{\epsilon}{\varepsilon}

\renewcommand{\bigcup}{\mathlarger{\bigcuppp}}
\renewcommand{\bigsqcup}{\mathlarger{\bigsqcuppp}}
\renewcommand{\sum}{\mathlarger{\sssum}} 

\let\originalleft\left 
\let\originalright\right
\renewcommand{\left}{\mathopen{}\mathclose\bgroup\originalleft}
\renewcommand{\right}{\aftergroup\egroup\originalright}

\newcommand{\sumprime}{\sideset{}{'}\sum}
\newcommand{\lv}{\left\Vert}
\newcommand{\dd}{\mathrm{d}}
\newcommand{\rv}{\right\Vert}
\newcommand{\lb}{\left[}
\newcommand{\rb}{\right]}
\newcommand{\lp}{\left(}
\newcommand{\rp}{\right)}
\newcommand{\lbr}{\left\lbrace}
\newcommand{\rbr}{\right\rbrace}
\newcommand{\lev}{\left\Vert}
\newcommand{\rev}{\right\Vert}

\newcommand{\hsp}{\hspace{0.1cm}}
\newcommand{\shsp}{\hspace{0.05cm}}
\newcommand{\R}{{\bf R}}
\newcommand{\Z}{{\bf Z}}

\newcommand{\SL}{\mathrm{SL}}
\newcommand{\SP}{\mathrm{Sp}}

\newcommand{\vol}{\operatorname{Vol}}
\newcommand{\diag}{\operatorname{diag}}
\newcommand{\cum}{\operatorname{cum}}
\newcommand{\0}{{\bf 0}}
\newcommand{\Y}{{Y_{2d}}}
\newcommand{\rvline}{\hspace*{-\arraycolsep}\vline\hspace*{-\arraycolsep}}
\newcommand\blfootnote[1]{%
  \begingroup
  \renewcommand\thefootnote{}\footnote{#1}%
  \addtocounter{footnote}{-1}%
  \endgroup
}
\numberwithin{equation}{section}
\titleformat
{\section} 
[display] 
{\scshape\large} 
{} 
{0.5ex} 
{\thesection. \hsp \hsp \centering} 
[] 
\titleformat
{\subsection} 
[display] 
{\scshape} 
{} 
{0.5ex} 
{\thesubsection. \hsp} 
[] 
\titleformat
{\subsubsection} 
[display] 
{\scshape} 
{} 
{0.5ex} 
{\thesubsubsection. \hsp} 
[] 
\newtheoremstyle{sltheorem}
{1.2em}                
{1.2em}                
{\slshape}        
{}                
{\scshape}       
{.}               
{.5em}               
{}                
\theoremstyle{sltheorem}

\newtheorem{theorem}{Theorem}
\newtheorem{definition}[theorem]{Definition}
\newtheorem{prop}[theorem]{Proposition}
\newtheorem{lemma}[theorem]{Lemma}
\numberwithin{theorem}{section}

\definecolor{shadethmcolor}{cmyk}{0,0,0,0}
\definecolor{shaderulecolor}{cmyk}{0,0,0,1} 

\allowdisplaybreaks

\begin{document}
\thispagestyle{empty}
\begin{center}
\noindent
\textbf{A CENTRAL LIMIT THEOREM FOR COUNTING FUNCTIONS \\ RELATED TO SYMPLECTIC LATTICES AND BOUNDED SETS}\\ \vspace{0.3cm}
\textsc{Kristian Holm}\footnote{\textsc{Mathematisches Seminar, Christian-Albrechts-Universität zu Kiel, 24118 Kiel, Germany}} \quad \quad \today \\
\end{center} 
\begin{adjustwidth}{2em}{2em}
\begin{small}
\textbf{Abstract.} We use a method developed by Björklund and Gorodnik to show a central limit theorem (as $T$ tends to $\infty$) for the counting functions $\# \lp \Lambda \cap \Omega_T \rp$ where $\Lambda$ ranges over the space $\Y$ of symplectic lattices in $\R^{2d}$ ($d \geqslant 4$). Here $\lbr \Omega_T \rbr_T$ is a certain family of bounded domains in $\R^{2d}$ that can be tessellated by means of the action of a diagonal semigroup contained in $\SP (2d, \R)$. In the process we obtain new $L^p$ bounds on a certain height function on $\Y$ originally introduced by Schmidt.
\end{small}
\end{adjustwidth}
\section{Introduction}
The interplay between lattices and subsets of their ambient Euclidean space has always been a central topic in the geometry of numbers. For example, Minkowski's classic lattice point theorem \cite[Theorem III.2.2]{cassels} gives conditions under which it is always possible to find non-zero lattice points in convex, centrally symmetric sets, and in the opposite direction the Minkowski-Hlawka theorem \cite[Corollary VI.3.2]{cassels} gives conditions on a symmetric set ensuring that at least some lattice will have to avoid it. \par In 1945, Siegel generalized the Minkowski-Hlawka theorem in his seminal paper \cite{siegel} and arguably created a kind of probabilistic geometry of numbers where questions about lattices and Euclidean sets can be studied from the perspective of a \textit{random lattice}. Namely, if $\Omega \subset \R^{n}$ is any measurable set, Siegels mean value theorem \cite{siegel} says that on average, a unimodular lattice $\Lambda$ of dimension $n$ has $\vol (\Omega)$ non-zero points in common with $\Omega$. (Such a statement of course presupposes the existence of a measure on the set $X_n$ of unimodular lattices in dimension $n$. See \textsc{Section 2} for more details.) In the same vein, Rogers \cite{rogers55} later generalized Siegel's theorem and proved formulas for the first $k$ moments ($k < n$) of the counting functions $\# (\Lambda \cap \Omega)$. \par In light of Siegel's theorem, a natural question is to what extent the function $\# (\Lambda \cap \Omega)$ can vary around its mean as $\Omega$ varies across some prescribed family of sets. In studying this problem, Schmidt employed Rogers' moment formula and observed \cite{schmidt60} that for almost all lattices (not necessarily unimodular), the counting functions $\# (\Lambda \cap \Omega_T)$ (where $\lbr \Omega_T \rbr$ is an increasing family of sets whose volumes constitute an unbounded sequence) enjoy rather sharp bounds. Essentially, Schmidt proved \cite[Theorem 1]{schmidt60} that for almost all lattices,
 \begin{align}\label{schmidtcalledit}
\# \lp \Lambda \cap \Omega_T \rp \cdot d \lp \Lambda \rp = \vol  \lp \Omega_T \rp + O \lp d \lp \Lambda \rp \vol \lp \Omega_T \rp^{1/2} \lp \log \frac{ \vol \lp \Omega_T \rp}{d(\Lambda)} \rp^2 \rp, \quad T \longrightarrow \infty,
\end{align} 
where $d (\Lambda) := \vol \lp \R^n / \Lambda \rp$ denotes the covolume of $\Lambda$. (Schmidt's result is slightly more general than this, but we simplified it for the sake of the present exposition.)\par 
Continuing with the probabilistic perspective, we note that (\ref{schmidtcalledit}) can be understood as something between a \textit{law of large numbers} (giving the main term) and a \textit{law of the iterated logarithm} (giving an error term). Indeed, suppose for simplicity that we specialize to unimodular lattices and partition the set $\Omega_T$ into $N=N(T)$ pieces of equal volume $V$: $\Omega_T \approx S_1 \sqcup \cdots \sqcup S_N$. By taking indicator functions of these sets, we then expect to have
 \begin{align*}
\# \lp \Lambda \cap \Omega_T \rp \approx \sum_{m = 1}^{N} \# \lp \Lambda \cap S_m \rp.
\end{align*} 
If we assume that the individual counting functions $\# \lp \Lambda \cap S_m \rp$ are "sufficiently (pairwise) independent" as random variables on the space of lattices, the classic law of large numbers and the law of the iterated logarithm imply that we can expect (almost surely)
 \begin{align*}
\frac{1}{N} \sum_{m = 1}^{N} \# \lp \Lambda \cap S_m \rp = V + O_\varepsilon \lp N^{-1/2 + \varepsilon} \rp, \quad N \longrightarrow \infty,
\end{align*} 
and hence
 \begin{align*}
\# \lp \Lambda \cap \Omega_T \rp = \sum_{m = 1}^{N} \# \lp \Lambda \cap S_m \rp = N V + O_\varepsilon \lp N^{1/2 + \varepsilon} \rp = \vol \lp \Omega_T \rp + O_\varepsilon \lp \vol \lp \Omega_T \rp^{1/2 + \varepsilon} \rp
\end{align*} 
as $T \longrightarrow \infty$, a statement quite similar to (\ref{schmidtcalledit}). 
\par The fact that such a probabilistic interpretation of Schmidt's theorem is possible motivates the question if any other classic probabilistic theorems have equivalents in terms of lattice point counting. This was answered in the affirmative by Björklund and Gorodnik who proved central limit theorems (CLT's) in the case of unimodular lattices for the sequence $\# \lp \Lambda \cap \Omega_T \rp$; in the first case with the family $\lbr \Omega_T \rbr$ being defined in terms of certain conditions coming from Diophantine approximation \cite{BG2}, and in the second case for a different family of sets defined in terms of products of linear forms \cite{BG3}. These results are in accordance with the fact that Schmidt's result remains valid if one specializes to the nullset of all unimodular lattices, a fact that follows from Schmidt's original proof. \par 
In a recent paper \cite{kelmer}, Kelmer and Yu consider the even smaller set $\Y$ of all symplectic lattices in $\R^{2d}$ (see \textsc{Section 2} for the definition) and prove a symplectic version of (a special case of) the moment formula of Rogers \cite{rogers55}, proving a mean square bound for the discrepancy and verifying that Schmidt's metrical bound continues to hold. Before the work of Kelmer--Yu, Athreya and Konstantoulas \cite{athreya} also studied the case of symplectic lattices (in fact, the larger class of general symplectic lattices) and proved a Rogers-type second moment formula to obtain both metrical and mean square bounds as well. However, the methods used in \cite{athreya} and \cite{kelmer} are very different. \par 
A notable difference between the symplectic mean square bounds of \cite{athreya, kelmer} and the analogous bound in the case of unimodular lattices is that, in the case of unimodular lattices, Rogers' formula gives an exact value for the second moment of the discrepancy. By contrast, an exact value for the second moment of the discrepancy is not known in the symplectic case. This is due to a number of technical obstacles that arise from the symplectic structure and render the moment formulas more complicated and less explicit. On the other hand, Siegel's mean value theorem still holds on the space $\Y$, so the mean of the random variables $\# \lp \Lambda \cap \Omega_T \rp$ over $\Y$ is known, and even equals the mean value taken over the larger space of unimodular lattices. In light of the results from \cite{BG2, BG3}, these considerations motivate the study of the distribution of the counting functions $\Lambda \longmapsto \# \lp \Lambda \cap \Omega_T \rp$ where $\Lambda$ is a symplectic lattice chosen uniformly at random from $\Y$. In this article, our goal is to show that, just as in the special linear case discussed above, even in the symplectic case one has a central limit theorem for the sequence $\# \lp \Lambda \cap \Omega_T \rp$ where $\lbr \Omega_T \rbr$ is a specific family of sets. We now describe this family in more detail. \par 
Our concrete family of sets is given as
 \begin{align*}
\Omega_T = \left\lbrace (\textbf{x}, \textbf{y}) \in \R^{d} \times \R^d : 1 \leqslant \Vert \textbf{x} \Vert \cdot \Vert \textbf{y} \Vert \leqslant 2 \text{ and } 1 \leqslant \Vert \textbf{y} \Vert < T \right\rbrace,
\end{align*} 
where we emphasize that the bounds on $\Vert \textbf{x} \Vert \cdot \Vert \textbf{y} \Vert$ are completely arbitrary and could be replaced with any positive real numbers. We note that "such sets" appear in a natural way when studying statistical properties of Diophantine approximation, cf. \cite{BG2}. However, we make no use of any special properties of this family other than the following facts, valid for every $T > 0$:
\begin{itemize}
\item[1)] $\Omega_T$ admits an approximate tessellation of the form $\Omega_T \approx a \Omega_2 \cup a^2 \Omega_2 \cup \cdots \cup a^{m(T)} \Omega_2$ where $a$ is a symplectic diagonal matrix, and $m(T)$ is a suitable integer depending on $T$. In particular, for $T = 2^N$ ($N \geqslant 1$ an integer), $\Omega_T$ coincides exactly with such a union;
\item[2)] $\Omega_T$ is symmetric in the sense that $-\Omega_T = \Omega_T$ and does not meet arbitrarily small neighbourhoods of $0$. (The symmetry condition facilitates the use of the symplectic version of Rogers' formula \cite[Theorem 1]{kelmer}.)
\end{itemize}
Given this family, we prove the following theorem.
\begin{theorem}\label{maintheorem} 
Suppose $d \geqslant 4$, and let $\Lambda \in \Y$ be distributed according to the probability measure $\mu$ coming from the restriction of Haar measure on $\SP( 2d, \R)$ (see \textsc{Section 2}). Then, there exists $\sigma \geqslant 0$ such that as $T \longrightarrow \infty$,
 \begin{align*}
\frac{\# (\Lambda \cap \Omega_T) - \vol\left(\Omega_T \right)}{\vol \left( \Omega_T \right)^{1/2}} \Longrightarrow N \lp 0, \sigma^2\rp.
\end{align*} 
\end{theorem}
\noindent \textsc{Remarks.} \par  1) We suspect that this result also holds in the cases $d = 2, 3$. However, we are unable to prove it due to technical limitations of our methods: In the case $d = 2$, the $L^2$ bound in Lemma \ref{BGeq415} becomes trivial, and we lose control of the error term in approximating the Siegel transform $\widehat{\chi_2}$  (see \textsc{Section 2}) with a smooth, compactly supported function on $\Y$; in the case $d = 3$, the inequality (\ref{apr22-ineq}) is not satisfied for $r \geqslant 4$, and we are thus unable to show that the cumulants of order $4$ and higher vanish as $T \longrightarrow \infty$, which is required due to the CLT criterion of Fréchet and Shohat (see below). The case $d = 1$ is obscured by the fact that the counting function $\# \lp \Lambda \cap \Omega_T \rp$ is not square integrable in this case. \par 
2) We believe that the limiting variance is positive, but have not been able to prove any such result. This matter is intimately connected with subtle properties of an $L^2$-isometry $\iota$ that appears in Kelmer and Yu's symplectic version \cite[Thm. 1]{kelmer} of Rogers' theorem. See \textsc{Section} 5 for more details. \par 
3) If $\Lambda$ varies over the larger space of \textit{all unimodular lattices} of dimension $d \geqslant 4$, we are able to prove that the sequence $\# \lp \Lambda \cap \Omega_T \rp$ satisfies a CLT with a strictly positive variance. This contrast to the symplectic case has to do with the fact that the $L^2$ isometry that appears in Rogers' theorem is very simple and explicit compared to $\iota$.\\ \par  
In the proof of Theorem \ref{maintheorem} we follow the arguments in \cite{BG2} closely. Our strategy for proving the theorem for a general $T$ is to first prove it for $T = 2^N$ and then show that this special case implies the theorem in its full generality. To deal with the special case $T = 2^N$, we employ the \textit{method of cumulants} (Theorem \ref{cltcriterion}), which is a CLT criterion for sequences of bounded functions due to Fréchet and Shohat. However, we cannot apply Theorem \ref{cltcriterion} directly as the functions
 \begin{align*}
\Lambda \longmapsto \frac{\# (\Lambda \cap \Omega_{2^N}) - \vol\left(\Omega_{2^N} \right)}{\vol \left( \Omega_{2^N} \right)^{1/2}} \quad \quad \lp N \geqslant 1 \rp
\end{align*} 
are unbounded on $\Y$. To remedy this, we exploit the fact that thanks to the tessellation properties of the family $\lbr \Omega_T \rbr$, the statement of Theorem \ref{maintheorem} can be given, for $T = 2^N$, in terms of $\SP(2d, \R)$-translates of the Siegel transform $\widehat{\chi_2}$ of the indicator function $\chi_2$ of $\Omega_2$, and the fact that this function can be approximated by a family of smooth and bounded functions on the space $\Y$. Concretely, if $\phi \in C_c^\infty (\Y)$ and $a \in G$ are fixed, and 
 \begin{align}\label{smoothfcts}
\psi_m := \phi \circ a^m - \int_{\Y} \phi \hsp \dd \mu, \quad \quad F_N := \frac{1}{\sqrt{N}} \sum_{m = 0}^{N-1} \psi_m,
\end{align} 
we first prove that the sequence $\lbrace F_N \rbrace$ satisfies a central limit theorem. This result is given in Theorem \ref{cltforsmoothfcts}, and we prove it in \textsc{Section 4} by appealing to Theorem \ref{cltcriterion} as $F_N$ \textit{is} bounded for each $N$. Specifically, we use a quantitative correlation estimate \cite[Thm. 1.1]{BEG} due to Björklund, Einsiedler, and Gorodnik in combination with a combinatorial technique developed by Björklund and Gorodnik in \cite{BG1} in order to analyse the cumulants $\mathrm{cum}_r \lp F_N \rp$ and prove that they vanish when $N \longrightarrow \infty$ in accordance with the hypotheses of Theorem \ref{cltcriterion}. In terms of the probabilistic heuristic about Schmidt's result above, this correlation estimate expresses the "sufficient independence" required of the counting functions corresponding to different tiles in the tessellation of $\Omega_{2^N}$. \par
In \textsc{Section 5} we will then construct a smooth, compactly supported function $\phi$ which approximates $\widehat{\chi_2}$, which will allow us to extend the central limit theorem proved in \textsc{Section 4} to our case of interest. In the case of general unimodular lattices considered in \cite{BG2}, this approximation makes use of integrability properties of a certain height function introduced by Schmidt \cite{schmidt}, namely
 \begin{align*}
\alpha \lp \Lambda \rp := \sup_V \bigg\{ d \lp V \cap \Lambda \rp^{-1} : V \cap \Lambda \text{ } \mathrm{is} \text{ } \mathrm{a} \text{ } \mathrm{lattice} \text{ } \mathrm{in} \text{ } V \bigg\},
\end{align*} 
where $V$ runs over all non-zero subspaces of the ambient Euclidean space of $\Lambda$. However, at a first glance it is not obvious whether these properties continue to hold when one considers this height function on $\Y$. In order to adapt this method to the symplectic case, we therefore also prove the following theorem which we show is optimal.
\begin{theorem}\label{aaalphaint}
The function $\alpha$ belongs to $L^p \lp Y_{2d} \rp$ for $p \leqslant d$. Consequently, for any such $p$ and $L > 0$ one has the estimate 
 \begin{align*}
\mu \lp \lbr \Lambda \in \Y : \alpha \lp \Lambda \rp \geqslant L \rbr \rp \ll_p L^{-p}.
\end{align*} 
\end{theorem}
By taking $V$ to be a one-dimensional subspace in the definition of $\alpha$, we see that $\alpha ( \Lambda)$ is at least equal to the inverse length of the shortest non-zero vector in $\Lambda$. Mahler's compactness theorem \cite[Chap. V, Thm. IV]{cassels} therefore implies that the preimage of any compact subset of $\Y$ under the height function $\alpha$ is, itself, compact. The $L^p$ bounds provided by Theorem \ref{aaalphaint} therefore give us as much control over this function as one could hope for. Since the Siegel transform $\widehat{\chi_2}$ that we wish to truncate is bounded when $\alpha$ is small, the fact that $\alpha$ is only large on a small subset of $\Y$ is a key ingredient in constructing the approximation $\phi$ that allows us to reduce Theorem \ref{maintheorem} to Theorem \ref{cltforsmoothfcts}. We prove Theorem \ref{aaalphaint} in \textsc{Section 3}. 
\section{Preliminaries}
We let $J = J_{2d}$ be the standard skew-symmetric matrix
 \begin{align*}
J_{2d} = \lp 
\begin{matrix}
0 & I_d \\
-I_d & 0
\end{matrix} 
\rp
\end{align*} 
where $I_d$ denotes the $d \times d$ identity matrix. Any $2d \times 2d$ matrix $A$ that satisfies $A^\intercal J A = J$ is called a \textit{symplectic} matrix. We will denote by $G = \SP (2d, \R)$ the group of all real symplectic matrices of dimension $2d \times 2d$, or the \textit{real symplectic group}. (Some authors use a different skew-symmetric reference matrix than $J$ to define symplecticity. The group $G$ is independent of this choice.)\par 
In the usual way, $J$ gives rise to an anti-symmetric bilinear form on $\R^{2d}$ given by $\omega (\textbf{x}, \textbf{y} ) = \textbf{x}^\intercal J \textbf{y}$. If $\Lambda \subset \R^{2d}$ is a unimodular lattice, we say $\Lambda$ is \textit{symplectic} if the restriction of $\omega$ to $\Lambda \times \Lambda$ takes values in $\Z$. Denoting the space of all symplectic $2d$-dimensional lattices by $Y_{2d}$, one can show that the group $G$ acts transitively on $Y_{2d}$ with stabilizer $\Gamma = \SP(2d, \Z)$ (see Proposition \ref{GtransonSym}). Consequently, the space $Y_{2d}$ can be realized as the coset space $G/\Gamma$. As both $G$ and $\Gamma$ are unimodular groups, the quotient inherits a $G$-invariant measure $\mu$ from the Haar measure on $G$. Moreover, this measure is finite (see e.g. \cite{morris}), so by normalizing it to ensure $\mu (\Y) = 1$ we can realize the set of symplectic lattices as a probability space. \par
In the following we take $T = 2^N$ with $N \geqslant 1$ an integer. As mentioned in the introduction, and as will be proved in \textsc{Section 5}, this comes at no loss of generality. The advantage of specializing to the subsequence $2^N$ is that by doing so, we can conveniently tessellate the set $\Omega_T = \Omega_{2^N}$ and exploit this to describe the function $\# \lp \Lambda \cap \Omega_{T} \rp = \# \lp \Lambda \cap \Omega_{2^N} \rp$ as a sum of \textit{Siegel transforms}.
\subsection{Reformulation in Terms of the Siegel Transform}
We now define the Siegel transform of a compactly supported function on Euclidean space $\R^{n}$ and show that the function $\# \lp \Lambda \cap \Omega_{2^N} \rp$ may be expressed in terms of the Siegel transform $\widehat{\chi_2}$ of the characteristic function $\chi_2$ of $\Omega_2$. 
\begin{definition}
Let $f : \R^{n} \longrightarrow \R$ be a measurable function of compact support. Its Siegel transform is the function $\widehat{f}$ on the space $X_n$ of unimodular lattices in $\R^n$ given by
 \begin{align*}
\widehat{f} : X_{n} \longrightarrow \R , \quad \quad \widehat{f}(\Lambda) = \sum_{\textbf{v} \in \Lambda \atop \textbf{v} \neq 0} f(\textbf{v}).
\end{align*} 
\end{definition}
We note that the condition on the support of $f$ ensures that the function $\widehat{f}$ takes finite values and so is well-defined. However, even if $f$ is a bounded function, its Siegel transform $\widehat{f}$ is typically not bounded on $X_n$. We will return to this matter later. \par 
If $\chi_T$ denotes the indicator function of the set $\Omega_T$, we now have
 \begin{align*}
\# \lp \Lambda \cap \Omega_{2^N} \rp = \widehat{\chi_{2^N}}\lp \Lambda \rp.
\end{align*} 
Furthermore, $\Omega_{2^N}$ can be tessellated by means of $G$-translations of the set $\Omega_2$. More specifically, for $t > 0$ we let
 \begin{align*}
a_t = \lp \begin{matrix} e^{t}I_d & 0 \\
0 & e^{-t} I_d
\end{matrix} \rp \in G,
\end{align*} 
so that with $b = a_{\log 2}$,
 \begin{align}\label{omegadecomp}
\Omega_{2^N} = \bigsqcup_{m = 0}^{N-1} b^{-m}\lp \Omega_2 \rp.
\end{align} 
Taking the Siegel transform of the indicator functions of both sides, we get 
 \begin{align*}
\widehat{\chi_{2^N}} (\Lambda)= \sum_{\textbf{v} \in \Lambda \atop \textbf{v} \neq 0} \chi_{2^N} (\textbf{v}) = \sum_{\textbf{v} \in \Lambda \atop \textbf{v} \neq 0} \sum_{m = 0}^{N-1} \chi_{2} \lp b^m (\textbf{v} ) \rp = \sum_{m = 0}^{N-1} \sum_{\textbf{v} \in \Lambda \atop \textbf{v} \neq 0} \chi_{2} \lp b^m (\textbf{v} ) \rp = \sum_{m=0}^{N-1} \widehat{\chi_2} \circ b^m.
\end{align*} 
By the symplectic version of Siegel's mean value theorem \cite[Theorem 2]{mosko} and the $G$-invariance of $\mu$, we now obtain
 \begin{align*}
\vol \lp \Omega_{2^N} \rp = \int_{\Y} \sum_{m=0}^{N-1} \widehat{\chi_2} \lp b^m (\Lambda ) \rp \hsp \dd \mu(\Lambda) = \sum_{m = 0}^{N-1} \vol \lp \Omega_2 \rp = N \cdot \vol \lp \Omega_2 \rp.
\end{align*} 
It follows that, for $T = 2^N$, the convergence in distribution claimed in Theorem \ref{maintheorem} is equivalent to 
 \begin{align}\label{thm1equiv}
\frac{1}{\sqrt{N}} \lp \sum_{m = 0}^{N-1} \widehat{\chi_{2}} \circ b^m - N \cdot \vol \lp \Omega_2 \rp \rp \Longrightarrow N(0, c^2\sigma^2)
\end{align} 
for $c = \vol \lp \Omega_2 \rp^{1/2}$. Hence, our main objective now is to prove (\ref{thm1equiv}).
\subsection{Cumulants and the Criterion of Fréchet and Shohat}
In order to prove (\ref{thm1equiv}), we use a CLT criterion due to Fréchet and Shohat. First, we need to introduce the notion of a \textit{cumulant.} 
\begin{definition}\label{cumdef} Let $(X, \nu)$ be a probability space. For bounded and measurable functions $f_1, \ldots, f_r$ on $X$, their joint cumulant of order $r$ is
 \begin{align*}
\cum_r \lp f_1, \ldots, f_r \rp = \sum_{\mathcal{P}} (-1)^{\# \mathcal{P}-1} (\# \mathcal{P} - 1)! \prod_{I \in \mathcal{P}} \hspace{0.05cm} \int_X \prod_{i \in I} f_i \hsp \dd \nu,
\end{align*} 
where the first sum is taken over the set of all partitions $\mathcal{P}$ of the set $\lbrace 1, \ldots, r \rbrace$. \par If $\mathcal{Q}$ is a partition of $\lbr 1, \ldots, r \rbr$, the conditional joint cumulant of order $r$ of the functions $f_1, \ldots, f_r$ is defined as
 \begin{align*}
\cum_r \lp f_1, \ldots, f_r \mid \mathcal{Q} \rp = \sum_{\mathcal{P}} (-1)^{\# \mathcal{P} - 1} (\# \mathcal{P} - 1 )! \prod_{I \in \mathcal{P}} \prod_{J \in \mathcal{Q}}  \hspace{0.05cm} \int_{X} \prod_{i \in I \cap J} f_i \hsp \dd \nu 
\end{align*} 
where $\mathcal{P}$ ranges over all partitions of $\lbr 1, \ldots, r \rbr$. \par 
Finally, if $f$ is a bounded and measurable function on $X$, we let
 \begin{align*}
\cum_r \lp f \rp = \cum_r (f, \ldots, f).
\end{align*} 
\end{definition}
\noindent \textsc{Remark.} It is not hard to see that $\cum_r$ is an $r$-linear functional on the space of bounded, measurable functions on $X$. (See e.g. \cite[Prop. 4.2]{speed}.) \\\\
With this definition in place, we can state the criterion of Fréchet and Shohat.
\begin{theorem}[{\cite{frechetshohat}}]\label{cltcriterion} Let $(X, \nu)$ be a probability space. Let $\lbrace F_T \rbrace$ be a sequence of real-valued bounded measurable functions on $X$ such that 
 \begin{align*}
\int_{X} F_T \hsp \dd \nu = 0, \quad \quad \sigma^2 := \lim_{T \rightarrow \infty} \int_{X} F_T^2 \hsp \dd \nu < \infty.
\end{align*} 
Suppose that for all $r \geqslant 3$, $\cum_r \lp F_T \rp \longrightarrow 0$ as $T \longrightarrow \infty$. Then, for every $\xi \in \R$,
 \begin{align*}
\nu \lp \lbrace F_T < \xi \rbrace \rp \longrightarrow \frac{1}{\sqrt{2 \pi \sigma^2}} \int_{-\infty}^\xi \exp \lp -\frac{t^2}{2 \sigma^2} \rp \hsp \dd t
\end{align*} 
as $T \longrightarrow \infty$.
\end{theorem}
\subsection{The Symplectic Group and Symplectic Lattices}
In this subsection we study certain Lie theoretic aspects of the group $G$ and its relationship to the space of all symplectic lattices. Specifically, in order to be able to introduce an important family of Sobolev norms on the space $C_c^\infty \lp \Y \rp$ of compactly supported, smooth functions on $\Y$ in the next subsection, we will examine the Riemannian structure on $G$ and the differential action of its Lie algebra on $C^\infty (G)$. We begin by proving that the space of symplectic lattices $Y$ coincides with the quotient $G / \Gamma$, justifying a claim that was made earlier. 
\begin{prop}\label{GtransonSym}
The natural action of $G$ on $\Y$ is transitive with stabilizer $\Gamma$. Consequently, $\Y$ may be identified with $G / \Gamma$. 
\end{prop}
\begin{proof}
Let $\Lambda$ be a symplectic lattice. If $\textbf{v} \in \Lambda$ is non-zero, the set $\omega (\textbf{v}, \Lambda) \subset \Z$ is a non-zero ideal, hence equal to $n_\textbf{v} \Z$ for some positive integer $n_\textbf{v}$. Choose now a $\textbf{v} \in \Lambda$ that minimizes $n_\textbf{v}$. Then, since $\Lambda$ is discrete, there is necessarily some $\textbf{w} \in \Lambda$ such that $\omega (\textbf{v}, \textbf{w} ) = n_\textbf{v}$. By the choice of $\textbf{v}$ and $\textbf{w}$, we see that these vectors span the sublattice 
 \begin{align*}
\Pi := \mathrm{span}_{\R} \lbr \textbf{v}, \textbf{w} \rbr \cap \Lambda.
\end{align*} 
Indeed, clearly the integer span of $\textbf{v}$ and $\textbf{w}$ is a sublattice of $\Pi$. Furthermore, if $\Pi$ is not contained in $\mathrm{span}_{\Z} \lbr \textbf{v}, \textbf{w} \rbr$, then $\Pi$ must contain a rational, non-integer multiple of either $\textbf{v}$ or $\textbf{w}$, meaning that one of $\textbf{v}$ and $\textbf{w}$ is not a primitive lattice point of $\Lambda$. As this is impossible because of the minimality of $n_\textbf{v}$, the claim follows. We now let $\textbf{e}_1 = \textbf{v}$, $\textbf{f}_1 = \textbf{w}$, and $n_1 = n_\textbf{v}$. We can then complete $\lbr \textbf{e}_1, \textbf{f}_1 \rbr$ to a basis $\lbr \textbf{e}_1, \textbf{f}_1, \textbf{g}_3, \ldots, \textbf{g}_{2n} \rbr$ of $\Lambda$. It will be clear later that we in fact have $n_1 = 1$, so we will now proceed by ensuring orthogonality between $\textbf{e}_1$ (resp. $\textbf{f}_1$) and the remaining basis vectors. To this end, note that all the integers $\omega (\textbf{e}_1, \textbf{g}_i)$ and $\omega (\textbf{g}_i, \textbf{f})$ ($i = 3, \ldots, 2n$) have to be divisible by $n_1$ due to the minimality of $n_1$. This means that we can obtain a (possibly) new lattice vector by 
replacing $\textbf{g}_i$ with the integer combination
 \begin{align*}
\textbf{g}_i' := \textbf{g}_i - \frac{1}{n_1} \omega (\textbf{g}_i, \textbf{f}_1 ) \textbf{e}_1 - \frac{1}{n_1} \omega (\textbf{e}_1, \textbf{g}_i ) \textbf{f}_1.
\end{align*}  
Since the map taking $\lbr \textbf{e}_1, \textbf{f}_1, \textbf{g}_3, \ldots, \textbf{g}_{2n} \rbr$ to $\lbr \textbf{e}_1, \textbf{f}_1, \textbf{g}'_3, \ldots, \textbf{g}'_{2n} \rbr$ has determinant $1$, we see that $\lbr \textbf{e}_1, \textbf{f}_1, \textbf{g}_3', \ldots, \textbf{g}_{2n}' \rbr$ is a basis for $\Lambda$. Moreover, the construction ensures that $\omega (\textbf{e}_1, \textbf{g}_i' ) = \omega (\textbf{g}_i', \textbf{f}_1 ) = 0$ for all $i$. \par 
Continuing inductively with $\Lambda / \mathrm{span} \lbr \textbf{e}_1, \textbf{f}_1 \rbr$ in place of $\Lambda$, and so on, we thus obtain a set $\lbr \textbf{e}_1, \ldots, \textbf{e}_n, \textbf{f}_1, \ldots, \textbf{f}_n \rbr$ of lattice vectors. The fact that each pair of lattice vectors forms a basis for the sublattice of $\Lambda$ containing their integer span means that $\textbf{e}_1, \ldots, \textbf{e}_n, \textbf{f}_1, \ldots, \textbf{f}_n$ form a basis for $\Lambda$. We now let
 \begin{align*}
M = \lb \begin{matrix} \textbf{e}_1 & \cdots & \textbf{e}_n & \textbf{f}_1 & \cdots & \textbf{f}_n \end{matrix} \rb,
\end{align*} 
which has determinant $\pm 1$. By construction, the matrix $M^\intercal J M$ has the form $\lp 
\begin{smallmatrix}
0 & N \\
-N & 0
\end{smallmatrix}
\rp$ with
 \begin{align*}
N = \lp 
\begin{matrix}
n_1 & 0 & \cdots & 0 \\
0 & n_2 & \cdots & 0 \\
\vdots & \vdots & \ddots & \vdots \\
0 & 0 & \cdots & n_d
\end{matrix}
\rp.
\end{align*} 
It follows that 
 \begin{align*}
\pm 1 = \det M = \pm \prod_{i = 1}^n n_{i}^2.
\end{align*} 
Therefore, as all the $n_i$ are positive integers, we have $n_{i} = 1$ for all $i$. Therefore we even have $M^\intercal J M = J$, which means that our matrix $M$ is symplectic. \par It remains to show that the stabilizer of $\Z^{2d}$ is $\Gamma$. However, since we have
 \begin{align*}
\mathrm{stab}_{\SL (2d, \R)}\lp \Z^{2d} \rp = \SL(2d, \Z),
\end{align*} 
we immediately obtain
 \begin{align*}
\mathrm{stab}_{G}\lp \Z^{2d} \rp = \mathrm{stab}_{\SL (2d, \R)}\lp \Z^{2d} \rp \cap G = \Gamma.
\end{align*} 
This completes the proof.
\end{proof} 
The Lie group $G = \SP (2d, \R)$ carries a right-invariant Riemannian metric that descends onto the quotient space $G / \Gamma$. Moreover, since $G$ is connected, this metric even gives rise in a natural way to a right-invariant distance function $\rho_G$ on $G$, and hence a distance function on $G/ \Gamma$ which we denote by $\rho_X$. 


\begin{definition}\label{opnorm}
Let $\Vert \cdot \Vert$ be the norm on $\mathfrak{g} = \mathrm{Lie}(G)$ coming from the Riemannian metric on $G$. Then we define the operator norm of $g$ as
 \begin{align*}
\Vert g \Vert_\mathrm{op} = \max \lbr \left\Vert g X g^{-1} \right\Vert : X \in \mathfrak{g}, \, \Vert X \Vert = 1 \rbr.
\end{align*} 
\end{definition}
Our operator norm of $g$ is therefore the usual operator norm of the adjoint map $\mathrm{Ad}(g) : \mathfrak{g} \longrightarrow \mathfrak{g}$. It is not hard to see that $\Vert \cdot \Vert_\mathrm{op}$ is submultiplicative. In what follows, however, we also need to be able to estimate $\Vert \cdot \Vert_\mathrm{op}$ from below.
\begin{lemma}\label{adjointeigenvalues}
For any non-zero $t \in \R$ there exists $\lambda = \lambda \lp t, d \rp > 1$ such that for any positive integer $m$,
 \begin{align*}
\left\Vert a_t^m \right\Vert_\mathrm{op} \geqslant \lambda^m.
\end{align*} 
\begin{proof}
It is a standard fact that the Lie algebra $\mathfrak{g}$ consists of all matrices $X$ such that $X^\intercal J + J X = 0$ where $J$ is the standard skew-symmetric matrix. Thus, with
 \begin{align*}
E = \begin{pmatrix}
  0
  & \rvline & \begin{matrix}
  0 &\cdots &1\\
  \vdots &\ddots &\vdots\\
  1 &\cdots &0
\end{matrix}   \\
\hline
  0 & \rvline &
  0
\end{pmatrix},
\end{align*} 
where all blocks have size $d \times d$, and the upper rightmost block only has non-zero entries in the corners on the antidiagonal, the matrices $E$ and $E^\intercal$ are elements of $\mathfrak{g}$ and eigenvectors of $\mathrm{Ad}(a_t)$ with respective eigenvalues $e^{2t}$ and $e^{-2t}$. If $t > 0$, let $\lambda = e^{2t} > 1$ and let $F = E / \Vert E \Vert$, and if $t < 0$, let $\lambda = e^{-2t} > 1$ and $F = E^\intercal / \Vert E^\intercal \Vert$. Then, if $m$ is any positive integer, we see that
 \begin{align*}
\lv a_t^m \rv_\mathrm{op} \geqslant \lv \mathrm{Ad}\lp a_t^m \rp F \rv = \lambda^m \lv F \rv = \lambda^m > 1,
\end{align*} 
and the claim follows.
\end{proof}
\end{lemma}
\par There is an action of $\mathfrak{g}$ on the space $C_c^\infty \lp G \rp$ of smooth, compactly supported functions on $G$ by means of the exponential map $\exp : \mathfrak{g} \longrightarrow G$. Namely, for $Y \in \mathfrak{g}$ with $\Vert Y \Vert = 1$ and $\phi \in C_c^\infty \lp G \rp$,
 \begin{align*}
\lp Y.\phi \rp (g) := \lim_{t \rightarrow 0} \frac{\phi \lp \exp \lp t Y \rp g \rp - \phi \lp g \rp}{t}.
\end{align*} 
Thus the element $Y \in \mathfrak{g}$ acts on $C_c^\infty (G)$ as a first-order differential operator which we denote by $\mathcal{D}_Y$. This differential action extends to an action of the universal enveloping algebra $\mathcal{U}(\mathfrak{g})$ on $C_c^\infty (G)$ given by
 \begin{align*}
\lp Y_1^{e_1} \cdots Y_r^{e_r} \rp . \phi := \mathcal{D}_{Y_1}^{e_1} \cdots\mathcal{D}_{Y_r}^{e_r} \phi := \mathcal{D}_Y \phi
\end{align*} 
for a monomial $Y = Y_1^{e_1} \cdots Y_r^{e_r}$, and given for a general element of $\mathcal{U} ( \mathfrak{g} )$ by extending the above definition linearly. 
\subsection{Norms on $C_c^\infty (G)$}
A crucial ingredient in the proof of the CLT for the sequence (\ref{smoothfcts}) is the correlation estimate given in Theorem \ref{BEGmain}. Anticipating this theorem, we now introduce a family of Sobolev norms on the space $C_c^\infty \lp \Y \rp$.
\begin{definition}\label{sobnorm}
For any integer $q \geqslant 1$ and for $\phi \in C_c^\infty \lp \Y \rp$, we define the norms
 \begin{align}
\big\Vert \phi \big\Vert_{C^q} &:= \max \Bigg\{ \big\Vert \mathcal{D}_Z \phi \big\Vert_\infty : {\small \begin{array}{l}
Z \in \mathcal{U}(\mathfrak{g}) \text{ is a monomial}\\
\deg Z \leqslant q
\end{array}} \Bigg\}, \label{cq-norm}\\
S_q \lp \phi \rp &:= \shsp \max \Big\{ \big\Vert \phi \big\Vert_\infty, \, \big\Vert \phi \big\Vert_{C^q} \Big\} \label{sq-norm}.
\end{align} 
\end{definition}
We will need the following two standard properties of the $S_q$-norm. 
\begin{lemma}\label{propertiesofsq-norm}
The family $\lbr S_q \rbr_q$ has the following two properties.
\begin{itemize}
\item[i)] For any $\phi \in C_c^\infty \lp \Y \rp$ and $g \in G$, 
 \begin{align}\label{N3}
S_q \lp \phi \circ g  \rp \ll_q \Vert g \Vert_{\mathrm{op}}^q S_q \lp \phi \rp.
\end{align} 
\item[ii)] For any $\phi_1, \phi_2 \in C_c^\infty \lp \Y \rp$, 
 \begin{align}\label{N4}
S_q \lp \phi_1 \phi_2 \rp \ll_q S_{q} \lp \phi_1 \rp S_{q} \lp \phi_2 \rp.
\end{align} 
\end{itemize}
\end{lemma}
\begin{proof}
The proof of this is long, but straightforward, and is therefore left to the reader.
\end{proof}
We will also need a family of Sobolev norms on Euclidean space. 
\begin{definition}
For any integer $q \geqslant 1$ and $f \in C_c^\infty \lp \R^{2d} \rp$, we let
 \begin{align*}
\lv f \rv_{C^q} := \max \lbr \lv f^{(\tau)} \rv_\infty : \deg \tau \leqslant q \rbr
\end{align*} 
where $\tau = (\tau_1, \ldots, \tau_{2d})$ denotes a multi-index, $\deg \tau = \tau_1 + \cdots + \tau_{2d}$, and 
 \begin{align*}
f^{(\tau)} =  \frac{\partial^{\tau_1}}{\partial x_{1}^{\tau_1}} \cdots \frac{\partial^{\tau_{2d}}}{\partial x_{2d}^{\tau_{2d}}} f.
\end{align*} 
\end{definition}
\section{$L^p$ Bounds on the Height Function $\alpha$}
In this second preliminaries section, we will investigate the height function $\alpha$ mentioned in the introduction that plays a key role in our construction of a smooth function that approximates the Siegel transform $\widehat{\chi_2}$. \par 
As we mentioned earlier, the Siegel transform of a bounded function $f$ on $\R^{2d}$ will generally be unbounded on $\Y$. Fortunately, however, if $f$ has compact support it is possible to remedy this situation due to an explicit connection between the Siegel transform $\hat{f}$ and the following function.
\begin{definition}\label{alphafunctiondefinition}
Let $\Lambda$ be a lattice (not necessarily unimodular) in any number of dimensions and denote by $d \lp \Lambda \rp < \infty$ the covolume of $\Lambda$. Then we define
 \begin{align*}
\alpha \lp \Lambda \rp := \sup_V \bigg\{ d \lp V \cap \Lambda \rp^{-1} : V \cap \Lambda \text{ } \mathrm{is} \text{ } \mathrm{a} \text{ } \mathrm{lattice} \text{ } \mathrm{in} \text{ } V \bigg\},
\end{align*} 
where $V$ runs over all non-zero subspaces of the ambient Euclidean space of $\Lambda$.
\end{definition} 
Note that one always has $1 / s \lp \Lambda \rp \leqslant \alpha \lp \Lambda \rp$ where $s \lp \Lambda \rp$ denotes the length of the shortest non-zero vector of $\Lambda$. Therefore, since the existence of very short lattice vectors will generally cause Siegel transforms to blow up, one could expect that $\alpha$ should more or less determine the growth of Siegel transforms on $X_n$. This result, due to Schmidt, is the source of our interest in $\alpha$.
\begin{prop}[{\cite[Lemma 2]{schmidt}}]\label{schmidt-alpha}
Suppose $f : \R^{n} \longrightarrow \R$ is a bounded function with compact support. Then for any unimodular lattice $\Lambda \in X_n$,
 \begin{align*}
\left| \hat{f} \lp \Lambda \rp \right| \ll_{\mathrm{supp} \, f} \Vert f \Vert_\infty \cdot \alpha \lp \Lambda \rp,
\end{align*} 
where $\widehat{f}$ denotes the Siegel transform of $f$.
\end{prop}
For the purposes of estimating norms of Siegel transforms on the space of unimodular lattices, Proposition \ref{schmidt-alpha} is particularly useful since $\alpha \in L^p (X_n)$ for $p = 1, \ldots, n-1$ {\cite[Lemma 3.10]{BG2-ref4}}. The goal of this section is to prove Theorem \ref{aaalphaint}, which extends this integrability result to the symplectic case. \par 
In order to prove Theorem \ref{aaalphaint}, we will need to describe a rather explicit Siegel set containing a fundamental domain of the coset space $G / \Gamma$. To this end, we initially describe three distinguished subgroups of $G$ and introduce some notation. \par Let ${\bf K} = \text{SO}(2d) \cap G$, and let
 \begin{align*}
{\bf A} = \left\lbrace \lp \begin{matrix}
D & 0 \\
0 & D^{-1}
\end{matrix} \rp : D > 0 \text{ diagonal} \right\rbrace.
\end{align*}  
Moreover, let $\textbf{N}(d) = \textbf{N}$ be the subgroup of $(2d)\times (2d)$ symplectic matrices given by
 \begin{align*}
\textbf{N}(d) := \lbr \lp \begin{matrix}
N & M \\
0 & N^{-\intercal} \end{matrix} \rp : N \text{ unipotent, upper-triangular, }\, NM^\intercal = MN^\intercal \rbr.
\end{align*} 
Finally, for real parameters $t, u > 0$, let ${\bf N}_{u} := \lbr A \in {\bf N}  : \Vert A \Vert_\infty \leqslant u \rbr$ and ${\bf N}_{\Z} := \mathrm{GL}(2d, \Z) \cap {\bf N}$, and define
\begin{align*}
{\bf A}_t := \Bigg\{ \diag \lp a_1, \ldots, a_d, a_1^{-1}, \ldots, a_d^{-1} \rp : 0 < a_i \leqslant t a_{i+1} \, \text{ for } i = 1, \ldots, d-1, \text{ and } 0 < a_d \leqslant t \Bigg\}.
\end{align*}
For future reference, we note that for $a \in \textbf{A}_t$, one has the bounds
 \begin{align}\label{boundsona}
a_i \leqslant t^{d+1-i}, \quad \quad i = 1, \ldots, d.
\end{align} 
\par We also note the following Iwasawa decomposition of the symplectic group.
\begin{theorem}[{\cite[Thm. 6.46]{knapp}}]\label{iwasawadec}
The Lie group $G$ is diffeomorphic to the product ${\bf K} \times {\bf A} \times {\bf N}$. 
\end{theorem}
With these definitions in place, we can describe the Siegel set containing a fundamental domain for $G / \Gamma$.
\begin{prop}\label{siegeldroid}
There exists an explicit $u = u(d) > 0$ such that, with $t = 2 / \sqrt{3}$, one has $\SP(2d, \R) = {\bf K} {\bf A }_t {\bf N}_u \cdot \SP(2d, \Z)$.
\end{prop}
We believe that such a result is known to experts of the field, but we have been unable to find a suitable reference. We therefore postpone the proof of Proposition \ref{siegeldroid} to Appendix \ref{siegelappendix}.
\subsection{Proof of Theorem \ref{aaalphaint}}
\noindent Assuming Proposition \ref{siegeldroid}, we will now prove Theorem \ref{aaalphaint}.
\par Because of the identification $\Y \simeq G / \Gamma$, we can view $\alpha$ as a right-$\Gamma$-invariant function on $G$. We will also denote by $\alpha$ the lift of this map to $G$, so that $\alpha \lp g \Gamma \rp = \alpha (g)$. Since the group of diagonal symplectic matrices can be parametrized in a very straightforward manner, making integration over (a subset of) this group relatively simple, we want to prove that up to some constant, for $g = k a n \in \textbf{K} \textbf{A}_t \textbf{N}_u$, $\alpha (g \Gamma )$ essentially only depends on $a$. To this end, we will need the following alternative characterization of $\alpha$.
\begin{lemma}\label{alphadifferentcharacterization}
Given a discrete subgroup $\Delta \leqslant \R^{2d}$, let $d ( \Delta ) \in \lp 0, \infty \rp$ denote the covolume of $\Delta$ in the subspace $V_\Delta \subset \R^{2d}$ spanned by $\Delta$. For $g \in G$, one has 
 \begin{align*}
\alpha (g) = \sup \lbr d \lp \Delta \rp^{-1} : \Delta \leqslant g \Z^{2d} \text{ is discrete} \rbr.
\end{align*} 
\end{lemma}
\begin{proof}
Let $g \in G$. Since $\alpha (g)$ is already defined as a supremum over all discrete subgroups $\Delta$ of a particular form, namely those satisfying $\Delta = V_\Delta \cap g \Z^{2d}$, we simply have to ensure that the remaining discrete subgroups of $g \Z^{2d}$ have larger covolumes than those considered in the definition of $\alpha$. However, this is clear: If $\Delta \leqslant g \Z^{2d}$ is a discrete subgroup with $\Delta$ \textit{properly contained} in $V_\Delta \cap g \Z^{2d}$, then 
 \begin{align*}
d (\Delta) = \left| V_\Delta / \Delta \right| > \left| V_\Delta / \lp V_\Delta \cap g \Z^{2d} \rp \right| = d \lp V_\Delta \cap g \Z^{2d} \rp,
\end{align*}  
so that $d \lp \Delta \rp^{-1} < d \lp V_\Delta \cap g \Z^{2d} \rp^{-1}$. 
\end{proof}
\begin{lemma}\label{alphabound}
There exists $C = C(t,u) > 0$ such that for any $g = kan \in {\bf K } {\bf A }_t {\bf N}_u$, one has $\alpha (g) \leqslant C \alpha (a)$.
\end{lemma}
\begin{proof}
From \cite[Lemma V.5.6]{bekkamayer} we see that there exists a function $\beta : G \longrightarrow \R_+$ with the property that if $\Delta \leqslant \R^{2d}$ is discrete, then for any $g \in G$,
 \begin{align*}
d \lp g \Delta \rp \leqslant \beta (g) d (\Delta).
\end{align*} 
Letting $\Delta ' = g^{-1} \Delta$, we obtain
 \begin{align*}
d ( \Delta )  = d \lp g \Delta' \rp \leqslant \beta (g)  \, d \lp \Delta ' \rp = \beta (g) \, d \lp g^{-1} \Delta \rp,
\end{align*} 
and hence, by exchanging $g$ with $g^{-1}$,
 \begin{align}\label{lowerboundforg}
d \lp g \Delta \rp \geqslant \beta \lp g^{-1} \rp^{-1} d (\Delta)
\end{align} 
for all $g \in G$.\par 
Note that for $a = \text{diag} \lp a_1, \ldots, a_d, a_1^{-1}, \ldots, a_d^{-1} \rp \in \textbf{A}$ and
 \begin{align*}
n = \lp \begin{matrix}
N & M \\
0 & N^{-\intercal}
\end{matrix} \rp \in \textbf{N}, \quad \quad N = (n_{ij}), \quad \quad M = (m_{ij}),
\end{align*} 
one has $an = n' a$ with
 \begin{align*}
n' = \lp \begin{matrix}
N' & M' \\
0 & (N')^{-\intercal}
\end{matrix} \rp, \quad \quad M' = (a_i a_j m_{ij}),
\end{align*} 
and with $N' = (n_{ij}')$ unipotent and upper-triangular with entries $n_{ij}' = a_i a_j^{-1} n_{ij}$ for $i < j$. In particular, for $a \in \textbf{A}_t$ and $n \in \textbf{N}_u$, the compactness of ${\bf N}_u$ and ${\bf N}_u^{-1}$ and (\ref{boundsona}) show that the entries of $n'$ are bounded. \par 
Since any discrete subgroup $\Delta \leqslant n' a \Z^{2d}$ has the form $\Delta = n' \Delta'$ for some discrete subgroup $\Delta ' \leqslant a \Z^{2d}$, it follows from (\ref{lowerboundforg}) and Lemma \ref{alphadifferentcharacterization} that for $g = kan \in {\bf K } {\bf A }_t {\bf N}_u$,
 \begin{align*}
\alpha (g) &= \sup \lbr d(\Delta)^{-1} : \Delta \leqslant n' a \Z^{2d} \text{ discrete} \rbr \\
&= \sup \lbr d (n' \Delta')^{-1} : \Delta' \leqslant a \Z^{2d} \text{ discrete} \rbr  \\
&\leqslant \beta \lp (n')^{-1} \rp \sup \lbr d (\Delta)^{-1} : \Delta \leqslant a \Z^{2d} \text{ discrete} \rbr \\
&= \beta \lp (n')^{-1} \rp \alpha (a).
\end{align*} 
From the proof of \cite[Lemma V.5.6]{bekkamayer} we see that $\lp \beta (x) \rp^2$ is a polynomial in the entries of $x \in G$.  Therefore we find that 
 \begin{align*}
\beta \lp (n')^{-1} \rp \leqslant \sup \lbr \beta  \lp n^{-1} \rp : n \in \textbf{N}_u \rbr = C < \infty,
\end{align*} 
which proves the lemma.
\end{proof}
\begin{lemma}\label{shapeoflattice}
Let $a \in {\bf A}_t$ and $r \in \lbr 1, \ldots, 2d \rbr$. Among all the rank $r$ subgroups of $a \Z^{2d}$, the group of the smallest covolume is the integer span of $r$ distinct columns in $a$.
\end{lemma}
\begin{proof}
Let $\Delta = \text{span}_{\Z} \lbr \textbf{x}_1, \ldots, \textbf{x}_r \rbr$ for $\textbf{x}_1, \ldots, \textbf{x}_r \in a \Z^{2d}$. We will write $a_{i} = a_{i-d}^{-1}$ for $i \geqslant d+1$, so that $a = \text{diag}\lp a_1, \ldots, a_d, a_{d+1}, \ldots, a_{2d} \rp$. For $i = 1, \ldots, r$, we can write
 \begin{align*}
\textbf{x}_i = x_{1,i} a_1 \textbf{e}_1 + x_{2,i} a_2 \textbf{e}_2 + \cdots + x_{2d,i} a_{2d} \textbf{e}_{2d}
\end{align*} 
where all $x_{i,j}$ are integers and $\textbf{e}_1, \textbf{e}_2, \ldots$ denote the standard basis vectors of $\R^{2d}$. Let us denote by $\textbf{X}_{i_1, \ldots, i_r}$ the $r \times r$ matrix obtained from the $2d \times r$ matrix
 \begin{align*}
\lp \begin{matrix}
x_{1,1} &x_{1,2} &\cdots &x_{1,r} \\
x_{2,1} &x_{2,2} &\cdots &x_{2,r} \\
\vdots &\vdots &\ddots &\vdots \\
x_{2d, 1} &x_{2d,2} &\cdots &x_{2d,r} 
\end{matrix} \rp
\end{align*} 
by removing all rows except those numbered $i_1, i_2, \ldots$ or $i_r$. Then we find that
 \begin{align*}
\text{covol} (\Delta)^2 = \Vert \textbf{x}_1 \wedge \cdots \wedge \textbf{x}_r \Vert^2 = \sum_{1 \leqslant i_1 < \cdots < i_r \leqslant 2d } \lp \prod_{j = 1}^r a_{i_j}^2 \rp \lp\det \textbf{X}_{i_1, \ldots, i_r} \rp^2,
\end{align*} 
and since $\textbf{x}_1 \wedge \cdots \wedge \textbf{x}_r \neq 0$, this number is at least equal to the product of the squares of the $r$ smallest numbers in $\lbr a_1^{\pm 1}, \ldots, a_d^{\pm 1} \rbr$. This proves the lemma.
\end{proof}
\begin{lemma}\label{haarmeasure}
Let $\dd n$ be a left (and right) Haar measure on ${\bf N}$. Then 
 \begin{align*}
\dd(an) = \lp \prod_{i = 1}^{d} a_i^{2(d-i)+1} \rp \, \dd n \, \dd a_1 \, \cdots \, \dd a_d
\end{align*} 
is a right Haar measure on ${\bf AN}$.
\end{lemma}
\begin{proof}
Since 
 \begin{align*}
\lp \begin{matrix}
D & 0 \\
0 & D^{-1}
\end{matrix} \rp 
\lp \begin{matrix}
N & M \\
0 & N^{-\intercal}
\end{matrix} \rp 
\lp \begin{matrix}
D^{-1} & 0 \\
0 & D
\end{matrix} \rp = 
\lp \begin{matrix}
D N D^{-1} & DMD \\
0 & \lp D N D^{-1} \rp^{- \intercal}
\end{matrix} \rp,
\end{align*} 
we see that $\textbf{A}$ normalizes $\textbf{N}$. When viewing ${\bf N}$ as a subset of Euclidean space, we therefore find that $\mathrm{Ad}(a)$ is a linear map on ${\bf N}$ with determinant 
 \begin{align*}
\det \mathrm{Ad} (a) = \rho \lp a \rp = \lp \prod_{1 \leqslant i < j \leqslant d} a_i a_j^{-1} \rp \lp \prod_{1 \leqslant j \leqslant i \leqslant d} a_i a_j \rp = \prod_{i = 1}^d a_i^{2(d-i)+2},
\end{align*} 
where the appearance of the second product is due to the fact that the entries $m_{ij}$ of $M$ with $i < j$ are dependent on the entries with $i \geqslant j$. Therefore, if $f \in C_c ({\bf AN})$ and $a_0 n_0 \in {\bf AN}$, the fact that ${\bf A}$ is abelian implies that  
 \begin{align*}
\int_{\bf A} \int_{\bf N} f(an a_0 n_0 ) \rho (a) \, \dd n \, \dd a 
&= \int_{\bf A} \int_{\bf N} f \lp a a_0 \lb a_0^{-1} n a_0 \rb  n_0 \rp \rho (a) \, \dd n \, \dd a \\
&= \int_{\bf A} \int_{\bf N} f\lp a \lb a_0^{-1} n a_0 \rb  n_0 \rp \rho \lp a_0^{-1} a \rp \, \dd n \, \dd a \\
&= \int_{\bf A} \int_{\bf N} f\lp a n n_0 \rp \rho (a_0) \rho \lp a_0^{-1} a \rp \, \dd n \, \dd a \\
&= \int_{\bf A} \int_{\bf N} f(an ) \rho (a) \, \dd n \, \dd a,
\end{align*} 
where the third equality is due to the change of variables $n \mapsto a_0 n a_0^{-1}$ with determinant $\rho (a_0)$. Since we have
 \begin{align*}
da = \prod_{i = 1}^n \frac{d a_i}{a_i},
\end{align*} 
the lemma follows.
\end{proof}
We are now ready to prove Theorem \ref{aaalphaint}.
\begin{proof}[Proof of Theorem \ref{aaalphaint}]
For $r = 1, \ldots, 2d$, let us introduce the function $\alpha_r$ on $G$ given by
\begin{align*}
\alpha_r \lp g \rp = \sup \lbr \left| V / ( V \cap g \Z^{2d} ) \right|^{-1} : {\small \begin{array}{l}
V \cap g \Z^{2d} \text{ is a lattice in a sub-}\\
\text{space $V \subset \R^{2d}$, $\text{dim } V = r$}
\end{array}} \rbr.
\end{align*}
To prove the claim, it is enough to show that for arbitrary $r$, the function $\alpha_r$ belongs to $L^p \lp G / \Gamma \rp$ for the mentioned values of $p$. To this end, in view of Lemma \ref{alphabound}, Lemma \ref{shapeoflattice}, and Lemma \ref{haarmeasure} it is enough to show that 
 \begin{align}\label{showthis}
\int_{{\bf A}_t} \prod_{i = 1}^r \mathrm{max}_i \lbr a_1^{\pm p}, \ldots, a_d^{\pm p} \rbr  \prod_{i = 1}^{d} a_i^{2(d-i)+1} \hsp \dd a_1 \cdots \dd a_d < \infty,
\end{align} 
where $\mathrm{max}_i$ denotes the $i$'th largest element of the set in question. \par Observe that as functions of $a_1, \ldots, a_d$
 \begin{align}\label{neatlittletrick}
\prod_{i = 1}^r \mathrm{max}_i \lbr a_1^{\pm p}, \ldots, a_d^{\pm p} \rbr < \sum a_1^{e_1 p} a_2^{e_2 p} \cdots a_d^{e_d p} < 3^d \prod_{i = 1}^r \mathrm{max}_i \lbr a_1^{\pm p}, \ldots, a_d^{\pm p} \rbr,
\end{align} 
the sum extending over all $d$-tuples $\lp e_1, \ldots, e_d \rp \in \lbr 0, \pm 1 \rbr^d$ with exactly $\min \lbr r, 2d-r \rbr$ non-zero entries. We therefore let $a_1^{e_1 p} a_2^{e_2 p} \cdots a_d^{e_d p}$ be an arbitrary monomial in the above sum and show that 
 \begin{align}\label{arbitrarymonomial}
\int_{{\bf A}_t} a_1^{e_1 p} a_2^{e_2 p} \cdots a_d^{e_d p} \prod_{i = 1}^{d} a_i^{2(d-i)+1} \hsp \dd a_1 \cdots \dd a_d < \infty.
\end{align} 
This will imply (\ref{showthis}).\par 
By the definition of ${\bf A}_t$, we see that the left-hand side of (\ref{arbitrarymonomial}) equals
 \begin{align*}
&\int_{{\bf A}_t} \prod_{i = 1}^{d} a_i^{e_i p + 2(d-i)+1} \hsp \dd a_1 \cdots \dd a_d 
\\ &\quad = \int_0^t \int_0^{t a_d} \int_0^{t a_{d-1}} \cdots \int_0^{t a_2} \prod_{i = 1}^{d} a_i^{e_i p + 2(d-i)+1} \hsp \dd a_1 \cdots \dd a_d \\
&\quad \ll_t \int_0^t \int_0^{t a_d} \int_0^{t a_{d-1}} \cdots \int_0^{t a_3} a_2^{(e_1 + e_2)p + 4d- 3} \prod_{i = 3}^{d} a_i^{e_i p + 2(d-i)+1} \hsp \dd a_2 \cdots \dd a_d,
\end{align*} 
if we have $e_1 p + 2(d-1)+1 \geqslant 0$. Analogously, if $(e_1 + e_2) p + 4d-3 \geqslant 0$, then the right-hand side is 
 \begin{align*}
\ll_t \int_0^t \int_0^{t a_d} \int_0^{t a_{d-1}} \cdots \int_0^{t a_4} a_3^{(e_1 + e_2 + e_3)p + 6d- 7 } \prod_{i = 4}^{d} a_i^{e_i p + 2(d-i)+1} \hsp \dd a_3 \cdots \dd a_d.
\end{align*} 
Continuing inductively, assuming that all the successively resulting exponents
 \begin{align*}
e_1 p + 2d - 1, \quad (e_1 + e_2)p + 4d - 3, \quad (e_1 + e_2 + e_3)p + 6d - 7, \quad \ldots
\end{align*} 
of $a_1, a_2, \ldots, a_{i-1}$ are non-negative, we find that the integral with respect to $a_i$ converges if the exponent of $a_i$ is non-negative, i.e. if
 \begin{align}\label{nicecondition}
2id - i^2 + i - 1 + (e_1 + \cdots + e_i) p \geqslant 0.
\end{align} 
In particular, if (\ref{nicecondition}) holds for $i = 1, \ldots, d$, then we obtain (\ref{arbitrarymonomial}). Note that for any $i$ with $e_1 + \ldots + e_i \geqslant 0$, (\ref{nicecondition}) is definitely satisfied, so assume that $i \in  \lbr 1, \ldots, d \rbr$ is such that $e_1 + \cdots + e_i < 0$. Then (\ref{nicecondition}) is satisfied if and only if
 \begin{align*}
p \leqslant \left| e_1 + \cdots + e_i \right|^{-1} \lp 2id - i^2 + i -1 \rp.
\end{align*} 
Since $|e_1 + \cdots + e_i | \leqslant i$, this is certainly true if
 \begin{align*}
p \leqslant i^{-1} (2id - i^2 + i - 1).
\end{align*} 
This inequality is satisfied when $p = 1, \ldots, d$. This proves (\ref{arbitrarymonomial}), and the theorem follows.
\end{proof}
\textsc{Remark.} This result is optimal in the sense that $\alpha \not \in L^{p} \lp \Y \rp$ for $p = d+1$. This claim will follow if, for example, $\alpha_{d} \not \in L^{d+1} \lp \Y \rp$. This, in turn, will follow by (\ref{neatlittletrick}) if we can prove that with $\lp e_1, \ldots, e_d \rp = (-1, \ldots, -1)$,
 \begin{align*}
\int_{{\bf A}_t} a_1^{e_1 p} a_2^{e_2 p} \cdots a_d^{e_d p} \prod_{i = 1}^{d} a_i^{2(d-i)+1} \hsp \dd a_1 \cdots \dd a_d = \infty.
\end{align*} 
To this end, it is enough to see that the sequence of successively resulting exponents considered in the proof will have to contain numbers less than or equal to $-1$, and indeed, our assumptions imply that for $i = 1, \ldots, d$, the $i$'th such exponent is
 \begin{align*}
2id - i^2 + i - 1 + (e_1 + \cdots + e_i)p = 2id - i^2 + i - 1 - i(d+1) = - i^2 + id - 1,
\end{align*} 
which equals $-1$ for $i = d$. \par Theorem \ref{aaalphaint} is therefore natural in the following sense: As in {\cite[Lemma 3.10]{BG2-ref4}}, the integrability properties of $\alpha$ depend on the dimension of the Cartan subgroup ${\bf A}$ in the Iwasawa decomposition.
\section{A Central Limit Theorem for Functions in $C_c^\infty \lp \Y \rp$}
The goal of this section is to prove the following intermediate theorem, which states that we do indeed have a central limit theorem for averages $F_N$ of translations of a smooth and compactly supported function on $Y_{2d}$ as defined in (\ref{smoothfcts}). \par 
For the remainder of this article, we will use the notation 
 \begin{align*}
\mu \lp \phi \rp := \int_{\Y} \phi \hsp \dd \mu,
\end{align*} 
where $\phi$ is an integrable function on $\Y$. We will also write 
 \begin{align*}
\mathrm{Leb}(f) := \int_{\R^{2d}} f(\textbf{x}) \, \dd \textbf{x}
\end{align*} 
where $f$ is an integrable function on $\R^{2d}$.
\begin{theorem}\label{cltforsmoothfcts}
Let $\phi \in C_c^\infty \lp \Y \rp$, let $a = \mathrm{diag} \lp e^t, \ldots, e^t, e^{-t}, \ldots, e^{-t} \rp$ for some $t > 0$, and define 
 \begin{align*}
\psi_m := \phi \circ a^m - \mu \lp \phi \rp, \quad \quad F_N := \frac{1}{\sqrt{N}} \sum_{m = 0}^{N-1} \psi_m.
\end{align*} 
Then there is $\sigma \geqslant 0$ such that as $N \longrightarrow \infty$, 
 \begin{align*}
F_N \Longrightarrow N\lp 0, \sigma^2 \rp.
\end{align*} 
\end{theorem}
\noindent \textsc{Remark.} In the event that $\sigma = 0$, the resulting distribution $N(0,0)$ is to be understood as the Dirac distribution at $0$.\\ \par 
It is immediate from the definition of $F_N$ that it integrates to $0$. Hence, using Theorem \ref{cltcriterion} to show that $F_N \Longrightarrow N\lp 0, \sigma^2 \rp$, we only need to demonstrate that 
 \begin{align}
\lim_{N \rightarrow \infty} \int_{\Y} F_N^2 \hsp \dd \mu &< \infty, \label{smoothf-finitevariance}
\end{align} 
and that
 \begin{align}
\lim_{N \rightarrow \infty} \cum_r \lp F_N \rp &= 0 \label{smoothf-cumtozero}
\end{align} 
for all $r \geqslant 3$. We note that (\ref{smoothf-finitevariance}) can be demonstrated without too much trouble as follows: One has 
 \begin{align*}
\int_{\Y} F_N^2 \hsp \dd \mu &= \frac{1}{N} \sum_{m = 0 \atop n = 0}^{N-1} \int_{\Y} \psi_{m-n}  \psi_0 \hsp \dd \mu = \frac{1}{N} \sum_{\pm s = 0}^{N-1} \lp N - \left|s \right| \rp \int_{\Y} \psi_s \psi_0 \hsp \dd \mu \\
&= \sum_{s \in \Z} \mathbbm{1} \lp 1-N \leqslant s \leqslant N-1 \rp \lp 1 - \frac{|s|}{N} \rp \int_{\Y} \psi_s \psi_0 \hsp \dd \mu.
\end{align*} 
The action $G \curvearrowright \Y$ is mixing with an exponential rate, cf. \cite[Thm. 1.1]{BEG} (see also Theorem \ref{BEGmain} below). Therefore, for any $N$, the series above is dominated termwise by an absolutely convergent series, in which case the theorem of dominated convergence shows that 
 \begin{align*}
\lim_{N \rightarrow \infty} \int_{\Y} F_N^2 \hsp \dd \mu = \sum_{s \in \Z} \, \int_{\Y} \psi_s \psi_0 \hsp \dd \mu = \sum_{s \in \Z} \lp \int_{\Y} \phi \cdot \lp \phi \circ a^s \rp - \mu \lp \phi \rp^2 \hsp \dd \mu \rp,
\end{align*} 
which is finite, again according to \cite[Thm. 1.1]{BEG}. This proves (\ref{smoothf-finitevariance}).
\subsection{Partitioning $r$-tuples of Natural Numbers}
In order to apply Theorem \ref{cltcriterion} to deduce Theorem \ref{cltforsmoothfcts}, all that remains is for us to demonstrate (\ref{smoothf-cumtozero}) for all $r \geqslant 3$. To this end, we will use the combinatorial tool given in Proposition \ref{decomposition}, due to Björklund and Gorodnik \cite{BG1}, which allows us to partition the natural numbers in a way that considerably facilitates our subsequent use of the quantitative correlation estimate given in Theorem \ref{BEGmain} below. \par 
Using the $r$-linearity of $\cum_r$, we find that 
 \begin{align}\label{cumalternative}
\cum_r \lp F_N \rp = \frac{1}{N^{r/2}} \sum_{\substack{m_1 = 0 \\ \cdots \\ m_r = 0}}^{N-1} \cum_r \lp \psi_{m_1}, \ldots, \psi_{m_r} \rp.
\end{align} 
We want to decompose the set over which the summation occurs using the following result, which is a special case of \cite[Prop. 6.2]{BG1} with $H = \Z$, cf. also \cite[Eq. (3.6)]{BG2}.
\begin{prop}[{\cite[Prop. 6.2]{BG1}}]\label{decomposition}
Suppose that $r \geqslant 3$ is an integer. Given $0 \leqslant \alpha < \beta$ and a partition $\mathcal{Q}$ of $\lbrace 1, \ldots, r \rbrace$, define
 \begin{align*}
\Delta(\alpha) = \lbr \textbf{s} \in \Z_+^{r} : |s_i - s_j | \leqslant \alpha \hsp \text{ for all $i,j$}\rbr
\end{align*} 
and
 \begin{align*}
\Delta_\mathcal{Q}\lp \alpha, \beta \rp &= \lbr \textbf{s} \in \Z_+^{r} : \max_{I \in \mathcal{Q}} \max_{i,j \in I} \lbr |s_i - s_j| \rbr \leqslant \alpha \, , \hsp \mathrm{ and } \hsp  
\min_{I, J \in \mathcal{Q} \atop I \neq J} \min_{i \in I \atop j \in J} \lbr |s_i - s_j| \rbr > \beta \rbr.
\end{align*} 
Then, given $0 = \alpha_0 < \beta_1 < \alpha _1 = (3+r) \beta_1 < \beta_2 < \cdots < \alpha_{r-1} = (3+r) \beta_{r-1} < \beta_{r}$, we have
 \begin{align*}
\Z_+^{r} = \Delta \lp \beta_{r} \rp \cup \lp \bigcup_{j = 0}^{r-1} \hsp \bigcup_{\# \mathcal{Q} \geqslant 2} \Delta_\mathcal{Q} \lp \alpha_j, \beta_{j+1} \rp \rp,
\end{align*} 
where the final union is taken over all partitions $\mathcal{Q}$ of $\lbr 1, \ldots, r \rbr$ with at least two parts.
\end{prop}
\noindent By intersecting the decomposition given by Proposition \ref{decomposition} with the set over which the summation in (\ref{cumalternative}) takes place, we get the sets
 \begin{align*}
\Omega \lp \beta_{r }, N \rp &:= \lbr 0, \ldots, N-1 \rbr^r \cap \Delta \lp \beta_{r } \rp, 
\\ 
\Omega_\mathcal{Q} \lp \alpha_j, \beta_{j+1}, N \rp &:= \lbr 0, \ldots, N-1 \rbr^r \cap \Delta_\mathcal{Q} \lp \alpha_j, \beta_{j+1} \rp,
\end{align*} 
and the decomposition
 \begin{align}\label{Ndecomposition}
\lbrace 0, \ldots, N-1 \rbrace^r = \Omega \lp \beta_{r }, N \rp \cup \lp \bigcup_{j = 0}^{r-1} \hsp \bigcup_{\# \mathcal{Q} \geqslant 2} \Omega_\mathcal{Q} \lp \alpha_j, \beta_{j+1}, N \rp \rp.
\end{align} 
\subsection{Estimating the Cumulants $\mathrm{cum}_r \lp F_N \rp$}
It follows from (\ref{Ndecomposition}) that the summands in (\ref{cumalternative}) can be partitioned into two categories, depending on where the corresponding indices $\textbf{m} = (m_1, \ldots, m_r)$ lie in the decomposition (\ref{Ndecomposition}). \par 
If we only allow the index $\textbf{m}$ in (\ref{cumalternative}) to run over $\Omega \lp \beta_{r }, N \rp$, the resulting contribution $A$ to $\cum_r (F_N)$ satisfies
 \begin{align}\label{shouldhavegonetolouisnielsen}
A \ll_r N^{1-r/2} \beta_{r }^{r-1} \left\Vert \phi \right\Vert_\infty^r.
\end{align} 
After all, there are $N$ ways to choose the largest coordinate of \textbf{m}, and the remaining $r-1$ coordinates are then confined to lying in an interval of length $\beta_{r }$, whence
 \begin{align*}
\# \Omega \lp \beta_{r }, N \rp \ll_r N \beta_{r }^{r-1}.
\end{align*} 
Moreover, the cumulant $\cum_r \lp \psi_{m_1}, \ldots, \psi_{m_r} \rp$ satisfies
 \begin{align}\label{cumest}
\nonumber \left| \cum_r \lp \psi_{m_1}, \ldots, \psi_{m_r} \rp \right| &\leqslant \sum_{\mathcal{P}} (\# \mathcal{P} - 1)! \prod_{I \in \mathcal{P}} \hspace{0.05cm} \int_{\Y} \prod_{i \in I} \left| \phi  \circ a^{m_i} - \mu \lp \phi \rp \right| \hsp \dd \mu \\
&\ll_r \sum_{\mathcal{P}} \prod_{I \in \mathcal{P}} \prod_{i \in I} \Vert \phi \Vert_\infty \ll_r \Vert \phi \Vert_\infty^r.
\end{align} 
\par Next, if $\textbf{m} \in \Omega_\mathcal{Q} \lp \alpha_j, \beta_{j+1}, N \rp$ with $\# \mathcal{Q} \geqslant 2$, we can estimate the resulting contribution to $\cum_r \lp \psi_{m_1}, \ldots, \psi_{m_r} \rp$ with the help of the following quantitative correlation estimate, which is a corollary of \cite[Thm. 1.1]{BEG}. 
\begin{theorem}\label{BEGmain}
Let $r \geqslant 2$ be an integer. Then there is a $\delta > 0$ such that for all $\phi_1, \ldots, \phi_r \in C_c^\infty \lp \Y \rp$ and $g_1, \ldots, g_r \in G$,
 \begin{align*}
\int_{Y_{2d}} \lp \phi_1 \circ g_1 \rp \cdots \lp \phi_r \circ g_r \rp \hsp \dd \mu = \prod_{i = 1}^r \lp \int_{\Y} \phi_i \hsp \dd \mu \rp + O_{q,r} \lp e^{-\delta \min_{i \neq j} \rho_G \lp g_i, g_j \rp} S_q \lp \phi_1 \rp \cdots S_q \lp \phi_r \rp \rp,
\end{align*} 
where the minimum is taken over all $(i,j) \in \lbr 1, \ldots r \rbr^2$ with $i \neq j$.
\end{theorem}
\noindent \textsc{Remark.} The original version \cite[Thm. 1.1]{BEG} of Theorem \ref{BEGmain} is formulated in terms of a different family of Sobolev norms. However, it is stated \cite[p. 6]{BEG} that the theorem may be formulated in terms of any family of norms satisfying the five conditions \cite[Eq. (1.9)-(1.13)]{BEG}. \\ \par 
Now, let $I \subset \lbr 1, \ldots, r \rbr$ be a non-empty subset, and suppose that $I$ contains numbers from exactly $k$ different sets in $\mathcal{Q}$. If $k = 1$, then we have
 \begin{align}\label{isconditionalcum1}
\int_{\Y} \prod_{i \in I} \shsp \psi_{m_i} \hsp \dd \mu = \prod_{J \in \mathcal{Q}} \shsp \int_{\Y} \prod_{i \in I \cap J} \psi_{m_i} \hsp \dd \mu,
\end{align} 
since in this case, there is exactly one choice of $J \in \mathcal{Q}$ such that $I \cap J \neq \varnothing$, and then we even have $I \cap J = I$. If $k \geqslant 2$, then (\ref{isconditionalcum1}) holds up to a small error. Indeed, if we let $Q_1, \ldots, Q_k$ be the non-empty sets in $\lbrace I \cap J : J \in \mathcal{Q} \rbrace$, then $I = Q_1 \sqcup \cdots \sqcup Q_k$, and therefore, if $f = \phi - \mu \lp \phi \rp$,
 \begin{align}\label{isconditionalcum2}
\nonumber \int_{\Y} \prod_{i \in I} \shsp \psi_{m_i} \hsp \dd \mu 
&= \int_{\Y} \prod_{i \in I} \hsp f \circ a^{m_i} \hsp \dd \mu 
= \int_{\Y} \prod_{j = 1}^k \lp \prod_{i \in Q_j} \shsp f \circ a^{m_i} \rp \hsp \dd \mu\\
&= \int_{\Y} \prod_{\ell = 1}^k \lp \prod_{i \in Q_\ell} \shsp f \circ a^{m_i-m_{Q_\ell}} \rp \circ a^{m_{Q_\ell}} \hsp \dd \mu,
\end{align} 
where we write $m_{Q_\ell} = \max \lbr m_i : i \in Q_\ell \rbr$ for $\ell = 1, \ldots, k$. For any such $\ell$ and for any $i \in Q_\ell$, let us write $n_{i,\ell} = m_i - m_{Q_\ell}$. Then, for each $\ell$,
 \begin{align*}
\prod_{i \in Q_\ell} \shsp f \circ a^{m_i-m_{Q_\ell}} = \sum_{K_\ell \subset Q_\ell} \lp -\mu \lp \phi \rp \rp^{\# K_\ell} g_{\ell,K_\ell}
\end{align*} 
with $g_{\ell, K_\ell} = \Pi_{i \not \in K_\ell} \phi \circ a^{n_{i,\ell}}$. We note that for each $\ell$ and $K_\ell \subset Q_\ell$, the function $g_{\ell,K_\ell}$ is in $C_c^\infty \lp \Y \rp$. Accordingly, (\ref{isconditionalcum2}) implies that
 \begin{align}\label{pullouterrorterm}
\nonumber \int_{\Y} \prod_{i \in I} \shsp \psi_{m_i} \hsp \dd \mu 
&= \int_{\Y} \prod_{\ell = 1}^k \lp \sum_{K_\ell \subset Q_\ell} \lp - \mu \lp \phi \rp \rp^{\# K_\ell}  g_{\ell,K_\ell} \rp \circ a^{m_{Q_\ell}} \hsp \dd \mu \\
\nonumber &= \int_{\Y} \sum_{K_\ell \subset \shsp Q_\ell \atop \text{for all $\ell$} } \lp - \mu \lp \phi \rp \rp^{\# K_1 + \cdots + \# K_k} \prod_{\ell = 1}^k \shsp g_{\ell, K_\ell} \circ a^{m_{Q_\ell}} \hsp \dd \mu 
\\[-2mm]
&= \sum_{K_\ell \subset \shsp Q_\ell \atop \text{for all $\ell$} } \lp - \mu \lp \phi \rp \rp^{\# K_1 + \cdots + \# K_k} \int_{\Y} \prod_{\ell = 1}^k \shsp g_{\ell, K_\ell} \circ a^{m_{Q_\ell}} \hsp \dd \mu.
\end{align} 
For any $\ell_1$ and $\ell_2$ with $\ell_1 \neq \ell_2$, $m_{Q_{\ell_1}}$ and $m_{Q_{\ell_2}}$ do not belong to the same part in $\mathcal{Q}$, and hence 
 \begin{align*}
\left| m_{Q_{\ell_1}} - m_{Q_{\ell_2}} \right| > \beta_{j+1}.
\end{align*} 
Let us assume that $M = m_{Q_{\ell_1}}-m_{Q_{\ell_2}} > 0$. Since $a = a_{t_0}$ for some $t_0 \in \R$, it now follows from \cite[Lemma 2.1]{BEG} and Lemma \ref{adjointeigenvalues} that there are numbers $\lambda = \lambda\lp t_0, d \rp > 1$, $C_1 \in \lp 0, 1 \rb$, and $C_2 > 0$ such that
 \begin{align*}
\rho_G \lp a^{m_{Q_{\ell_1}}}, a^{m_{Q_{\ell_2}}} \rp = \rho_G \lp 1, a^{M} \rp \geqslant C_1 \log \left\Vert a^M \right\Vert_\mathrm{op} -C_2 \geqslant C_1 M \log \lambda - C_2.
\end{align*} 
Therefore, with $\delta' = \delta C_1 \log \lambda > 0$, Theorem \ref{BEGmain} shows that the left-hand side of (\ref{pullouterrorterm}) is 
 \begin{align}\label{pullouterrorterm2}
\sum_{K_\ell \subset \shsp Q_\ell \atop \text{for all $\ell$} } \lp - \mu \lp \phi \rp \rp^{\# K_1 + \cdots + \# K_k} \lp \prod_{\ell = 1}^k \shsp \int_{\Y} g_{\ell,K_\ell} \hsp \dd \mu + \hsp O_{q,r} \lp e^{- \delta' \beta_{j+1}} \prod_{\ell = 1}^k \shsp S_q \lp g_{\ell,K_\ell} \rp \rp \rp.
\end{align} 
We want to pull the remainder term out of the sum. Of course, to do so, we must rid it of its dependence on the sets $K_1, \ldots, K_k$. We observe that for any fixed $\ell$ and $K_\ell$, if we write $t_\ell = \# \lp Q_\ell \setminus K_\ell \rp$, we have
 \begin{align*}
S_q \lp g_{\ell, K_\ell} \rp 
&= S_q \lp \prod_{i \in Q_\ell \setminus K_\ell} \phi \circ a^{m_i - m_{Q_\ell}} \rp 
\\
&\ll_q \prod_{i \in Q_\ell \setminus K_\ell} S_{q} \lp \phi \circ a^{m_i - m_{Q_\ell}} \rp \\
&\ll_q S_{q} \lp \phi \rp^{t_\ell} \prod_{i \in Q_\ell \setminus K_\ell} \left\Vert \lp a^{-1} \rp^{m_{Q_\ell}-m_i} \right\Vert_\mathrm{op}^{q} 
\\
&\ll S_{q} \lp \phi \rp^{t_\ell} \left\Vert a^{-1} \right\Vert_\mathrm{op}^{q \alpha_j t_\ell},
\end{align*} 
where we used (\ref{N3}), (\ref{N4}), the submultiplicativity of $\Vert \cdot \Vert_\mathrm{op}$, and the fact that for any $i \in Q_\ell \setminus K_\ell$,
 \begin{align*}
m_{Q_\ell} - m_i = \left| m_{Q_\ell} - m_i \right| \leqslant \alpha_j
\end{align*} 
by the choice of $m_{Q_\ell}$. We now find that
 \begin{align*}
\prod_{\ell = 1}^k \shsp S_q \lp g_{\ell, K_\ell} \rp 
&\ll_q \prod_{\ell = 1}^k \shsp S_{q } \lp \phi \rp^{t_\ell} \left\Vert a^{-1} \right\Vert_\mathrm{op}^{q \alpha_j t_\ell} \ll \max \lbr 1, S_{q } \lp \phi \rp^r \rbr \left\Vert a^{-1} \right\Vert_\mathrm{op}^{q \alpha_j r}.
\end{align*} 
Along with (\ref{pullouterrorterm}), this proves that with $\xi := \log \left\Vert a^{-1} \right\Vert_\mathrm{op} > 0$ and
 \begin{align*}
R &= e^{- \delta' \beta_{j+1}} \max \lbr 1, S_{q } \lp \phi \rp^r \rbr \left\Vert a^{-1} \right\Vert_\mathrm{op}^{q \alpha_j r} = \max \lbr 1, S_{q } \lp \phi \rp^r \rbr e^{- \lp \delta' \beta_{j+1} - q \xi \alpha_j r \rp },
\end{align*} 
one has
 \begin{align*}
\int_{\Y} \prod_{i \in I} \shsp \psi_{m_i} \hsp \dd \mu &= \sum_{K_\ell \subset \shsp Q_\ell \atop \text{for all $\ell$}} \lp - \mu \lp \phi \rp \rp^{\# K_1 + \cdots + \# K_k} \prod_{\ell = 1}^k \shsp \int_{\Y} g_{\ell,K_\ell} \hsp \dd \mu + O_{q,r} \lp R \rp \\[-1mm]
&= \prod_{\ell = 1}^k \shsp \sum_{K_\ell \subset Q_\ell} \lp - \mu \lp \phi \rp \rp^{\# K_\ell} \int_{\Y} g_{\ell,K_\ell} \hsp \dd \mu + O_{q,r} \lp R \rp  \\
&= \prod_{\ell = 1}^k \shsp \int_{\Y} \prod_{i \in Q_\ell} \lp f \circ a^{n_{i,\ell}} \rp \hsp \dd \mu + O_{q,r} \lp R \rp \\
&= \prod_{\ell = 1}^k \shsp \int_{\Y} \prod_{i \in Q_\ell} \lp f \circ a^{m_i} \rp \hsp \dd \mu +O_{q,r} \lp R \rp \\
&=  \prod_{\ell = 1}^k \shsp \int_{\Y} \prod_{i \in Q_\ell} \shsp \psi_{m_i} \hsp \dd \mu + O_{q,r} \lp R \rp
\end{align*} 
where we used the $G$-invariance of $\mu$ in the penultimate step. Along with (\ref{isconditionalcum1}) this proves that for any $\textbf{m} = (m_1, \ldots, m_r) \in \lbrace 0, \ldots, N-1 \rbrace^r$ with $\textbf{m} \in \Omega_\mathcal{Q} \lp \alpha_j, \beta_{j+1}, N \rp$ and $\# \mathcal{Q} \geqslant 2$, we have
 \begin{align}\label{isconditionalcum3}
\int_{\Y} \prod_{i \in I} \shsp \psi_{m_i} \hsp \dd \mu &= \prod_{J \in \mathcal{Q}} \shsp \int_{\Y} \prod_{i \in I \cap J} \psi_{m_i} \hsp \dd \mu + O_{q,r} \lp \max \lbr 1, S_{q } \lp \phi \rp^r \rbr e^{- \lp \delta' \beta_{j+1} - q \xi \alpha_j r \rp } \rp.
\end{align} 
This proves an approximate version of (\ref{isconditionalcum1}), as claimed.
\subsection{Final Estimates of the Cumulants}
To complete our estimate of $\cum_r \lp \psi_{m_1}, \ldots, \psi_{m_r} \rp$, we recall the following result from \cite{BG2}. 
\begin{prop}[{\cite[Prop. 3.5]{BG2}}]
For any partition $\mathcal{Q}$ of $\lbr 1, \ldots, r \rbr$ with $\# \mathcal{Q} \geqslant 2$, and for any $\psi_1, \ldots, \psi_r \in L^\infty (\Y)$, we have $\cum_r \lp \psi_{1}, \ldots, \psi_{r} \mid \mathcal{Q} \rp = 0$.
\end{prop}
Now, by summing the estimate (\ref{isconditionalcum3}) over all partitions $\mathcal{P}$ of $\lb r \rb$ and by letting $I$ denote an element of $\mathcal{P}$, we finally obtain that
 \begin{align*}
\cum_r \lp \psi_{m_1}, \ldots, \psi_{m_r} \rp 
&= \cum_r \lp \psi_{m_1}, \ldots, \psi_{m_r} \mid \mathcal{Q} \rp + O_{q,r} \lp \max \lbr 1, S_{q } \lp \phi \rp^r \rbr e^{- \lp \delta' \beta_{j+1} - q \xi \alpha_j r \rp } \rp \\
&\ll_{q,r} \max \lbr 1, S_{q } \lp \phi \rp^r \rbr e^{- \lp \delta' \beta_{j+1} - q \xi \alpha_j r \rp }
\end{align*} 
thanks to the proposition above. It now follows from (\ref{cumalternative}) and (\ref{shouldhavegonetolouisnielsen}) that we have
 \begin{align}\label{laststepbeforecumtozero}
\nonumber \cum_r \lp F_N \rp 
&= \frac{1}{N^{r/2}} \sum_{\substack{m_1 = 0 \\ \ldots \\ m_r = 0}}^{N-1} \cum_r \lp \psi_{m_1}, \ldots, \psi_{m_r} \rp \\
&\ll_{q,r,\phi}  N^{1-r/2} \beta_{r }^{r-1} + N^{r/2} \sum_{j = 0}^r e^{- \lp \delta' \beta_{j+1} - q \xi \alpha_j r \rp },
\end{align} 
where we used the trivial bound $\# \Omega_\mathcal{Q} \lp \alpha_j, \beta_{j+1}, N \rp \leqslant N^r$ for all $j$. \par 
Now we choose an explicit sequence $\beta_1, \ldots, \beta_{r }$ such that the right-hand side of (\ref{laststepbeforecumtozero}) goes to $0$ as $N \longrightarrow \infty$. Following \cite[Sect. 3.2.4]{BG2}, we reduce this problem to choosing a single parameter $\gamma > 0$ by defining $\beta_1 = \gamma$ and
 \begin{align}\label{choiceofbeta}
\beta_{j+1} = \max \lbr \gamma + (3+r)\beta_j, \, \gamma + \lp \delta ' \rp^{-1}r (3+r) q \xi \beta_j \rbr, \quad \quad j = 1, \ldots, r-1.
\end{align} 
This choice of $\beta_{j+1}$ ensures that $\alpha_j = (3+r) \beta_j < \beta_{j+1}$, as required in 
Proposition \ref{decomposition}. Additionally, we have
 \begin{align*}
\delta' \beta_{j+1} \geqslant \delta' \gamma + r(3+r) q \xi \beta_j = \delta '\gamma + q \xi r \alpha_j,
\end{align*} 
which implies that $\delta' \beta_{j+1} - q \xi \alpha_j r \geqslant \delta' \gamma > 0$. Furthermore, by induction, (\ref{choiceofbeta}) and the equalities
 \begin{align*}
\beta_1 = \gamma, \quad \quad \beta_2 = \gamma \cdot \max \lbr 4 + r, 1 + \lp \delta ' \rp^{-1} r (3+r) q \xi \rbr 
\end{align*} 
imply that $\beta_{r } \ll_{r,q} \gamma$. This and (\ref{laststepbeforecumtozero}) show that
 \begin{align*}
\cum_r \lp F_N \rp \ll_{q, r, \phi } N^{1-r/2} \gamma^{r-1} + N^{r/2} e^{- \delta' \gamma }.
\end{align*} 
Since $r \geqslant 3$, we have $1 - r/2 \leqslant - 1/2$. Hence we obtain (\ref{smoothf-cumtozero}), provided that 
 \begin{align*}
\gamma^{r-1} = o \lp N^{1/2} \rp, \quad \quad N = o \lp e^{2 \delta' \gamma / r } \rp.
\end{align*} 
This suggests taking $\gamma = r \log N / \delta '$, which indeed has the required properties. This proves (\ref{smoothf-cumtozero}) for all $r \geqslant 3$ and hence Theorem \ref{cltforsmoothfcts}.
\section{A Central Limit Theorem for $\widehat{\chi_2}$}
The main result of this section is that the function $\widehat{\chi_2}$ can be approximated by a smooth, compactly supported function $\phi$ on $\Y$, and that this approximation allows us to transfer the central limit theorem for the $F_N$ defined in (\ref{smoothfcts}) to the averages 
 \begin{align}\label{thisisthestuff}
G_N := \frac{1}{\sqrt{N}} \sum_{m = 0}^{N-1} \psi_m,
\end{align} 
where $\psi_m := \widehat{\chi_2} \circ b^m - \vol \lp \Omega_2 \rp$. We follow the arguments in \cite{BG2}. \par 
One of the reasons why Theorem \ref{BEGmain} cannot be applied directly to the function $\widehat{\chi_2}$ is that this is not a smooth function on $\Y$ because $\chi_2$ is not a smooth function on $\R^{2d}$. However, since the Lie derivatives we introduced in \textsc{Section} 2 commute with the operation of taking the Siegel transform (cf. (\ref{siegelcommuteswithdiff}) below), we are led to considering the Siegel transform of a smooth version $f_\varepsilon$ of $\chi_2$ instead and keeping track of the resulting error in such an approximation. While this will ameliorate the problem in question, there is still the obstacle that the resulting Siegel transform will not have compact support. To deal with this, we note that by Mahler's compactness theorem \cite[Chap. V, Thm. IV]{cassels} the function $\alpha$ that we introduced in \textsc{Section} 3 is proper. Hence, a first idea towards constructing $\phi$ would be to define it in terms of $\widehat{f_\varepsilon}$ and the indicator function of the set of lattices where $\alpha$ is \textit{small}. The resulting function will not be smooth, however, so we will work with "smooth indicator functions" $\eta_L$ of such sets $\alpha^{-1}(\lb 0, L \rb)$ instead. We now proceed to the details.
\subsection{Approximating $\widehat{\chi_2}$ with Compactly Supported $C^\infty$-Functions}
We first construct the function $f_\varepsilon$ and give some of its properties. Let $\theta : \R^{2d} \rightarrow \R$ be a smooth non-negative function that integrates to $1$ and has support contained in the ball of radius $1/2$ centered at the origin in $\R^{2d}$. Also, let 
 \begin{align}\label{fepsdef}
\theta_\varepsilon (\textbf{x}) := \varepsilon^{-2d} \theta \lp \varepsilon^{-1}\textbf{x} \rp, \quad \quad f_\varepsilon := \theta_\varepsilon * \mathbbm{1}_{\mathcal{M}(\varepsilon)},
\end{align} 
where $*$ denotes convolution and $\mathcal{M}(\varepsilon)$ is an $\varepsilon/2$-thickening of $\Omega_2$. Then $f_\varepsilon$ belongs to $C_c^\infty \lp \R^{2d} \rp$ and has support contained in an $\varepsilon$-neighbourhood of $\Omega_2$. In addition to this, we record the following properties of the family $\lbr f_\varepsilon \rbr$ for later use (cf. \cite[Sect. 6]{BG2}): 
 \begin{align}
\left\Vert f_\varepsilon - \chi_2 \right\Vert_{1} &\ll \varepsilon, \label{feps01} \\
\left\Vert f_\varepsilon - \chi_2 \right\Vert_2 &\ll \sqrt{\varepsilon}, \label{feps02}\\ 
\left\Vert f_\varepsilon \right\Vert_{C^q} &\ll \varepsilon^{-q}. \label{feps03}
\end{align} 
As mentioned above, the function $f_\varepsilon$ serves as a smooth indicator function for $\Omega_2$ in the following sense: For any $\varepsilon > 0$, we have $\chi_2 \leqslant f_\varepsilon \leqslant 1$. Indeed, it is straightforward that $f_\varepsilon \leqslant 1$, and if $\textbf{x}_0 \in \Omega_2$, then
 \begin{align*}
f_\varepsilon \lp \textbf{x}_0 \rp = \int_{\R^{2d}} \theta_\varepsilon \lp \textbf{y} \rp \mathbbm{1}_{\mathcal{M} (\varepsilon)} \lp \textbf{x}_0 - \textbf{y} \rp \hsp \dd \textbf{y} = \int_{-\mathcal{M}( \varepsilon) + \textbf{x}_0} \theta_\varepsilon ( \textbf{y} ) \hsp \dd \textbf{y} = 1,
\end{align*} 
since $\theta_\varepsilon$ has support contained in the ball $B_{\varepsilon / 2} (0) \subset - \mathcal{M}(\varepsilon) + \textbf{x}_0$.\par 
We now turn to the construction of the family of functions $\lbr \eta_L \rbr$ on $\Y$. As the following lemma shows, $\eta_L$ will serve as a smooth indicator function of a compact subset of $\Y$ depending on an arbitrary threshold parameter $c$.
\begin{lemma}[{\cite[Lemma 4.11]{BG2}}]\label{etaisgod}
For arbitrary $c > 1$, there is a family $\lbr \eta_L \rbr \subseteq C_c^\infty (\Y)$ such that, for any $L > 0$, the function $\eta_L$ takes values in $\lb 0, 1 \rb$ and satisfies $\Vert \eta_L \Vert_{C_q} \ll 1$. Additionally, $\eta_L$ has the following properties: \par If $\alpha (\Lambda) > cL$, then
 \begin{align}
\eta_L (\Lambda) = 0. \label{alphaone}
\end{align} 
\par If $\alpha (\Lambda) \leqslant c^{-1}L$, then
 \begin{align}
\eta_L (\Lambda ) = 1. \label{alphatwo}
\end{align} 
\end{lemma}
\noindent \textsc{Remarks.} \par 1) It should be noted that \cite[Lemma 4.11]{BG2} is a statement about the family $\lbr \eta_L \rbr$ on the larger space of all unimodular lattices. However, the proof extends verbatim to the current situation. \par 
2) The proof of \cite[Lemma 4.11]{BG2} shows that $\eta_L$ is given explicitly by
 \begin{align}\label{etadef}
\eta_L (\Lambda) := \lp \Xi * \mathbbm{1}_{\lbr \Lambda \in \Y : \alpha \lp \Lambda \rp \leqslant L \rbr}\rp (\Lambda) = \int_G \Xi (g) \mathbbm{1}_{\lbr \alpha \leqslant L \rbr} \lp g^{-1} \Lambda \rp \hsp \dd \mu (g),
\end{align} 
where $\Xi \in C_c^\infty \lp \Y \rp$ is a non-negative function depending on $c$ with $\mu ( \Xi ) = 1$. \\ \par  
We now let $\phi = \hat{f_\varepsilon} \eta_L$ so that $\phi$ is a compactly supported and smooth function that approximates $\chi_2$. Furthermore, we define
 \begin{align*}
G_N^{(\varepsilon, L)} := \frac{1}{\sqrt{N}} \lp \sum_{m = 0}^{N-1} \lp \widehat{f_\varepsilon} \eta_L \rp \circ b^m - N \int_{\Y} \widehat{f_\varepsilon} \eta_L \hsp \dd \mu \rp.
\end{align*} 
The following lemma shows how well $G_N^{(\varepsilon, L)}$ approximates $G_N$.
\begin{lemma}\label{smoothapprox}
For any $p = 1, \ldots, d$, as $N \longrightarrow \infty$, 
 \begin{align*}
\left\Vert G_N^{(\varepsilon, L)} - G_N \right\Vert_{1} \ll_{p,c} \sqrt{N}\lp L^{-p/2} + \epsilon \rp.
\end{align*} 
\end{lemma}
\begin{proof}
By the triangle inequality and the $G$-invariance of $\mu$, we have
 \begin{align*}
\left\Vert G_N^{(\varepsilon, L)} - G_N \right\Vert_{1} 
&\leqslant 2 \sqrt{N} \lv \widehat{f_\varepsilon} \eta_L - \widehat{\chi_2}  \rv_1 
\\ &\leqslant 2 \sqrt{N} \lp \lv \widehat{f_\varepsilon} \lp 1 - \eta_L \rp \rv_1 + \lv \widehat{f_\varepsilon} - \widehat{\chi_2} \rv_1 \rp.
\end{align*} 
We now estimate the $L^1$-norms on the right-hand side. By (\ref{alphatwo}), Proposition \ref{schmidt-alpha}, Theorem \ref{aaalphaint}, and the Cauchy-Schwartz inequality, we have 
 \begin{align*}
\lv \widehat{f_\varepsilon} \lp 1 - \eta_L \rp \rv_1 &\leqslant \int_{\lbr \alpha \geqslant c^{-1} L \rbr} \widehat{f_\varepsilon} \lp 1 - \eta_L \rp \hsp \dd \mu \leqslant \int_{\lbr \alpha \geqslant c^{-1} L \rbr} \widehat{f_\varepsilon} \hsp \dd \mu 
\\ &\leqslant \mu \lp \lbr \alpha \geqslant c^{-1} L \rbr \rp^{1/2} \lv \widehat{f_\varepsilon} \rv_2 \ll_{p} c^{p/2} L^{-p/2} \lv f_\varepsilon \rv_\infty  \lv \alpha \rv_2 
\\[+0.5em] &\ll_{c,p} L^{-p/2},
\end{align*} 
since $f_\varepsilon \leqslant 1$ and the family $\lbr \mathrm{supp} \, f_\varepsilon \rbr_\varepsilon$ is uniformly bounded.\par 
To estimate the second $L^1$-norm, we note that the inequality $f_\varepsilon \geqslant \chi_2$ is preserved by the Siegel transform, and hence
 \begin{align*}
\lv \widehat{f_\varepsilon} - \widehat{\chi_2} \rv_1 = \int_{\Y} \widehat{f_\varepsilon - \chi_2} \hsp \dd \mu = \lv f_\varepsilon - \chi_2 \rv_1 \ll \varepsilon
\end{align*}  
by Siegel's mean value theorem and (\ref{feps01}). This proves the lemma. 
\end{proof}
In light of Lemma \ref{smoothapprox}, we would like to take the parameters $L$ and $\varepsilon$ as suitable functions of $N$ so that as $N \longrightarrow \infty$, 
 \begin{align}\label{conditionsonepsandl}
\varepsilon = o \lp N^{-1/2} \rp, \quad \quad  N = o \lp L^p \rp
\end{align} 
for some $p = 1, \ldots, d$. Doing so, we observe that in order to prove Theorem \ref{maintheorem} in the case $T=2^N$, it is enough to prove that the sequence $G_N^{(\varepsilon, L)}$ satisfies a central limit theorem. Namely, our assumptions then imply that for some $\sigma \geqslant 0$,
 \begin{align}
G_N^{(\varepsilon, L)} &\Longrightarrow N\lp 0, \sigma^2 \rp, \label{hbo-1}\\
\left\Vert G_N^{(\varepsilon, L)} - G_N \right\Vert_{1} &\longrightarrow 0 \label{hbo-2},
\end{align} 
as $N \longrightarrow \infty$, which implies a central limit theorem for $G_N$ by standard methods. Hence, our goal now is to show that with $\varepsilon (N)$ and $L(N)$ chosen so that (\ref{conditionsonepsandl}) holds, we do indeed have convergence as in (\ref{hbo-1}). (Note that we do not automatically have a central limit theorem for $G_N^{(\epsilon, L)}$ by virtue of Theorem \ref{cltforsmoothfcts}. Although, for a fixed $N$, $\smash{G_N^{(\epsilon, L)}}$ has the form (\ref{smoothfcts}), the functions $\phi$ in the sum \textit{depend on} $N$ because $\varepsilon$ and $L$ depend on $N$, and consequently Theorem \ref{cltforsmoothfcts} does not apply, as this is a statement about the average behaviour of a \textit{fixed} function as $N$ tends to infinity.)
\subsection{Proof of Theorem \ref{maintheorem}}
We can prove a central limit theorem for $G_N^{(\epsilon, L)}$ by mimicking the proof of Theorem \ref{cltforsmoothfcts} and taking into account how the error terms behave when $\varepsilon \longrightarrow 0$ and $L \longrightarrow \infty$ in accordance with (\ref{conditionsonepsandl}). Appealing again to Theorem \ref{cltcriterion}, our main goal is therefore to prove that there is some choice of the parameters $\varepsilon$ and $L$ such that
 \begin{align}
\lim_{N \rightarrow \infty} \int_{\Y} \lp G_N^{(\epsilon, L)}\rp^2 \hsp \dd \mu \in \lb 0, \infty \rp, \label{GN-finitevariance}
\end{align} 
(in particular, the limit exists), and 
 \begin{align}
\lim_{N \rightarrow \infty} \cum_r \lp G_N^{(\epsilon, L)} \rp &= 0 \label{GN-cumtozero}
\end{align} 
for all $r \geqslant 3$. 
\subsubsection{Investigations of the Limiting Variance}
We now demonstrate that (\ref{GN-finitevariance}) holds. If $G_N$ is given by (\ref{thisisthestuff}), we note that due to the triangle inequalities,
 \begin{align*}
\big\| G_N \big\|_2 - \lv G_N - G_N^{(\epsilon, L)} \rv_2 \leqslant \lv G_N^{(\epsilon, L)} \rv_2 \leqslant \big\| G_N \big\|_2 + \lv G_N - G_N^{(\epsilon, L)} \rv_2,
\end{align*} 
and hence (\ref{GN-finitevariance}) will follow if we can prove that
 \begin{align}\label{aux01}
\lim_{N \rightarrow \infty} \big\| G_N \big\|_2 \in \lb 0, \infty \rp
\end{align} 
exists and prove that
 \begin{align}
\lv G_N - G_N^{(\epsilon, L)} \rv_2 \longrightarrow 0. \label{aux02}
\end{align} 
We therefore proceed by proving (\ref{aux01}) and (\ref{aux02}).\par 
To obtain (\ref{aux01}), we will need a theorem due to Kelmer and Yu that expresses integrals over $Y_{2d}$ of primitive Siegel transforms in terms of Euclidean integrals, i.e., a sympletic equivalent of a special case of Rogers' integration formula \cite[Thm. 4]{rogers55}. Since our Siegel transforms are not primitive, we state this version of their result as Theorem \ref{kelmerthm} below. \par 
We first introduce some notation. As in \cite[Sect. 1.3]{kelmer} we can write any $\textbf{x} \in \R^{2d} \setminus \lbr \0 \rbr$ in polar coordinates as $\textbf{x} = K a_y \textbf{e}_{2d}$ where
 \begin{align*}
K = \mathrm{SO}(2d) \cap \SP(2d, \R), \quad \quad a_y = \diag (y, 1, \ldots, 1, 1/y ) \quad (y > 0),
\end{align*} 
and $\textbf{e}_{2d}$ denotes the $(2d)$'th standard basis vector. Next, if $f$ is a bounded function on $\R^{2d}$ with compact support, we define the function
 \begin{align}\label{perioddef}
P_f (\textbf{x}) = P_f( Ka_y) = \int_{Y'} \int_{\lb 0, 1 \rp^{2d-1}} \sum_{\textbf{v} \in \Z^{2d} \setminus \lbr {\bf 0} \rbr} f (K a_y m(A) u_{\textbf{t},s} \textbf{v}) \hsp \dd \textbf{t} \hsp \dd s \hsp \dd \mu_{d-1} (A),
\end{align} 
where $Y' = \lbr m(A) : A \in \SP(2d-2, \R)/\SP(2d-2, \Z) \rbr \simeq Y_{2d-2}$, and
 \begin{align*}
m(A) = \lp \begin{matrix}
1 &0 &0 \\
0 &A &0 \\
0 &0 &1 
\end{matrix} \rp, \quad 
u_{\textbf{t},s} = \lp \begin{matrix}
1 &0 &0 \\
\textbf{t} &I &0 \\
s &\textbf{t}^* &1 
\end{matrix} \rp
\end{align*} 
with $\textbf{t}^* := \lp t_2, \ldots, t_{2d-1} \rp^* := \lp t_{2d-1}, \ldots, t_{d+1}, -t_d, \ldots, -t_2 \rp$.
\begin{theorem}[Corollary to {\cite[Theorem 1.1]{kelmer}}]\label{kelmerthm}
Suppose $f, \, g : \R^{2d} \rightarrow \R$ are even and bounded and have compact support. Then one has the formula
 \begin{align*}
\int_{Y_{2d}} \hat{f} \hat{g} \, \, \dd \mu = \frac{1}{\zeta(2d)} \sum_{j \geqslant 1} \int_{\R^{2d}} P_{f} (\textbf{x}) g(j\textbf{x}) \, \, \dd \textbf{x}.
\end{align*} 
\end{theorem}
\begin{proof}
Denote the primitive Siegel transform of a function $h$ by $\widetilde{h}$. It follows immediately from \cite[Remark 5.14]{kelmer} and \cite[Eq. (5.11)]{kelmer} that if $f_k$ and $g_j$ are even and bounded and have compact support, 
 \begin{align}\label{kelmersimpler}
\int_\Y \widetilde{f_k} \widetilde{g_j} \hsp \dd \mu = \frac{1}{\zeta (2d)} \int_{\R^{2d}} \mathcal{P}_{f_k} g_j \hsp \dd \textbf{x},
\end{align} 
where the operator $\mathcal{P}$ maps any bounded, compactly supported, measurable function $h$ on $\R^{2d}$ to the function $\mathcal{P}_h : \R^{2d} \setminus \lbr \0 \rbr \longrightarrow \R$, whose value at $\textbf{x} = K a_y \textbf{e}_{2d} \in \R^{2d} \setminus \lbr \0 \rbr$ is
\begin{align*}
\mathcal{P}_h (\textbf{x}) = \int_{Y'} \int_{\lb 0, 1 \rp^{2d-1}} \widetilde{h}\lp K a_y m(A) u_{\textbf{t},s} \Z^{2d} \rp \hsp \dd \textbf{t} \hsp \dd s \hsp \dd  \mu_{d-1} (A),
\end{align*}
cf. \cite[Eq. (2.2)]{kelmer} and the proof of \cite[Prop. 2.2]{kelmer}. If we now take $f_k (\textbf{x}) = f (k \textbf{x})$ and $g_j (\textbf{x}) = g(j \textbf{x})$ in (\ref{kelmersimpler}), the relation 
 \begin{align*}
\widehat{h}(\Lambda ) = \sum_{k \geqslant 1} \widetilde{h_k} (\Lambda) \quad (\text{$h$ even, bounded and with compact support)}
\end{align*} 
shows that 
 \begin{align*}
\int_\Y \widehat{f} \widehat{g} \hsp \dd \mu = \sum_{k ,j \geqslant 1} \int_\Y \widetilde{f_k} \widetilde{g_j} \hsp \dd \mu.
\end{align*} 
Hence, it suffices to show that 
 \begin{align*}
\sum_{k \geqslant 1} \mathcal{P}_{f_k} = P_f,
\end{align*} 
but this is clear from the definition of $\mathcal{P}$. 
\end{proof}
For later use, we record an additional corollary to \cite[Theorem 1.1]{kelmer}.
\begin{lemma}\label{siegelestnorm}
Suppose $f : \R^{2d} \rightarrow \R$ is bounded and has compact support. Then 
 \begin{align*}
\left\Vert \widehat{f} \right\Vert_2^2 \ll \Vert f \Vert_2^2 + \Vert f \Vert_1^2. 
\end{align*} 
\end{lemma}
\begin{proof}
Since the (primitive) Siegel transform only depends on the even part of a function, we can assume that $f$ is an even function. In this case the claim follows immediately from \cite[Theorem 1.1]{kelmer} and the fact that the map 
\begin{align*}
\iota : L^2_\mathrm{even} \lp \R^{2d} \setminus \lbr \0 \rbr \rp \longrightarrow L^2_\mathrm{even} \lp \R^{2d} \setminus \lbr \0 \rbr \rp, \quad \quad \iota(f) = -f + \frac{1}{2}\lp \mathcal{P}_f - \frac{1}{\zeta (2d)} \int_{\R^{2d}} f(\textbf{x}) \hsp \dd \textbf{x} \rp 
\end{align*}
is an $L^2$-isometry (cf. \cite[Eq. (5.11)]{kelmer} and \cite[Prop. 5.5]{kelmer}).
\end{proof}
We will also need the following result.
\begin{prop}\label{riemanntotherescue}
Let $f : \R^{2d} \rightarrow \R$ be a Riemann integrable function whose support is compact and does not contain the origin. If $\textbf{x} \in \R^{2d} \setminus \lbr \0 \rbr$ is written in polar coordinates as $\textbf{x} = K_\textbf{x} \Vert \textbf{x} \Vert \textbf{e}_{2d}$, one has
 \begin{align*}
P_f \lp \textbf{x} \rp = \sum_{n \neq 0} f \lp n \textbf{x} \rp  +  \Vert \textbf{x} \Vert^{-1} \sum_{n \in \Z} \int_{\R^{2d-1}} f \lp K_\textbf{x} \lp \Vert \textbf{x} \Vert^{-1} n, \textbf{r} \rp^\intercal \rp \hsp \dd \textbf{r}.
\end{align*} 
In particular, 
 \begin{align}\label{thereistherub}
P_f (\textbf{x}) = \int_{\R^{2d}} f(\textbf{r}) \hsp \dd \textbf{r} \hsp + O \lp \Vert \textbf{x} \Vert^{-1} \rp, \quad \quad \textbf{x} \in \R^{2d} \setminus \lbr \0 \rbr.
\end{align} 
\end{prop}
\textsc{Remark.} In order to prove that the limiting variance in Theorem \ref{maintheorem} is strictly positive, a better understanding of the remainder term in (\ref{thereistherub}) is necessary, cf. the computation before Lemma \ref{BGeq415} .
\begin{proof}
To ease the notation, we will write $K = K_\textbf{x}$ and $y = \Vert \textbf{x} \Vert^{-1}$. Take any non-zero $\textbf{v} = (v_1, \ldots, v_{2d})^\intercal \in \Z^{2d}$. If we write $\widetilde{\textbf{v}} := (0, v_2, \ldots, v_{2d-1}, 0)^\intercal$ and $\textbf{t} = (t_2, \ldots, t_{2d-1})^\intercal$, then we note that
 \begin{align*}
u_{\textbf{t},s}\textbf{v} = \lp \begin{matrix}
v_1 \\
t_2 v_1 + v_2 \\
\vdots \\
t_{2d-1} v_1 + v_{2d-1} \\
sv_1 + (0,\textbf{t}^*,0) \cdot \widetilde{\textbf{v}} + v_{2d}
\end{matrix} \rp.
\end{align*} 
We now consider two cases, and in each case we study the integral 
 \begin{align}\label{integral}
\int_{\lb 0, 1 \rp^{2d-1}}  f\lp K a_y m(A) u_{\textbf{t},s} \textbf{v} \rp \hsp \dd t_2 \hsp \cdots \hsp \dd t_{2d-1} \hsp \dd s.
\end{align} 
\par In the first case, assume that $v_1 = 0$. If $v_i \neq 0$ for some $i = 2, \ldots, 2d-1$, assume with no loss of generality that $v_2 > 0$. Since the integral converges absolutely, we can change the order of integration so that we first integrate with respect to the variable $t_{2d-1}$. Changing variables in this integral by letting $r = r(t_{2d-1}) = (0,\textbf{t}^*,0) \cdot \widetilde{\textbf{v}} + v_{2d}$, we find that the integral with respect to $t_{2d-1}$ equals
 \begin{align*}
v_2^{-1} \int_{(0,\textbf{t}^*,0) \cdot \widetilde{\textbf{v}} + v_{2d} - v_2 t_{2d-1}}^{(0,\textbf{t}^*,0) \cdot \widetilde{\textbf{v}} + v_{2d} - v_2 t_{2d-1} + v_2} f \lp K a_y m(A) (0, v_2, \ldots, v_{2d-1}, r)^\intercal \rp \hsp \dd r ,
\end{align*} 
where we stress that the bounds of the integral are independent of $t_{2d-1}$. By summing this over $v_{2d} \in \Z$, we see that the resulting domains of integration 
 \begin{align*}
\lb (0,\textbf{t}^*,0) \cdot \widetilde{\textbf{v}} + v_{2d} - v_2 t_{2d-1}, (0,\textbf{t}^*,0) \cdot \widetilde{\textbf{v}} + v_{2d} - v_2 t_{2d-1} + v_2 \rb
\end{align*} 
cover the real line exactly $v_2$ times. Therefore, this sum over $v_{2d} \in \Z$ equals 
 \begin{align}\label{yeah}
\int_{\R} f \lp K a_y m(A) \lp 0, v_2, \ldots, v_{2d-1}, r \rp^\intercal \rp \hsp \dd r.
\end{align} 
\par Next, if both $v_1 = 0$ and $\widetilde{\textbf{v}} = 0$, the integral (\ref{integral}) equals
 \begin{align}\label{thelittleonebehindyouismouse}
f \lp K a_y \lp 0, \ldots, 0, v_{2d} \rp^\intercal \rp = f\lp v_{2d} K a_y \textbf{e}_{2d} \rp = f \lp v_{2d} \textbf{x} \rp.
\end{align} 
We conclude that with $F_{K,y} \lp r_2, \ldots, r_{2d-1} \rp := \int_{\R} f \lp K a_y \lp 0, r_2, \ldots, r_{2d-1}, r_{2d} \rp^\intercal \rp \hsp \dd r_{2d}$,
 \begin{align}\label{conclusion1}
\nonumber &\sum_{v_{2d} \in \Z}\hsp  \int_0^1 \cdots \int_0^1 \int_0^1 f\lp K a_y m(A) u_{\textbf{t},s} \textbf{v} \rp \hsp \dd t_2 \cdots \hsp \dd t_{2d-1} \hsp \dd s \\
&\quad \quad = F_{K,y} \lp A \lp v_2, \ldots, v_{2d-1} \rp^\intercal \rp + \sum_{v_{2d} \in \Z \setminus \lbr 0 \rbr} f \lp v_{2d} \textbf{x} \rp
\end{align}  
whenever $v_1 = 0$.
\par In the second case, assume that $v_1 \neq 0$. Without loss of generality we then assume that $v_1 > 0$. Let $\textbf{r}(t_2, \ldots, t_{2d-1}, s) := (r_2, \ldots, r_{2d})^\intercal$ be given by the relation
 \begin{align*}
u_{\textbf{t},s} \textbf{v} = \lp \begin{matrix}
v_1 \\
r_2 \\
r_3 \\
\vdots \\
r_{2d}
\end{matrix} \rp,
\end{align*} 
so that $\dd \textbf{r} = v_1^{2d-1} \dd t_2 \cdots \dd t_{2d-1} \, \dd s$. Then we find that (\ref{integral}) is equal to
 \begin{align*}
v_1^{-(2d-1)} \int_{\textbf{r}\lp \lb 0, 1 \rp^{2d-1} \rp} f(K a_y m(A) \lp v_1,\textbf{r} \rp^\intercal) \hsp \dd \textbf{r}
\end{align*} 
where 
 \begin{align*}
\textbf{r}\lp \lb 0, 1 \rp^{2d-1} \rp = \prod_{i = 2}^{2d-1} \lb v_i, v_i + v_1 \rp \times \lb (0,\textbf{t}^*,0) \cdot \widetilde{\textbf{v}} + v_{2d}, (0,\textbf{t}^*,0) \cdot \widetilde{\textbf{v}} + v_{2d} + v_1 \rp.
\end{align*} 
When summing this integral over all $v_2, \ldots, v_{2d} \in \Z$, we see that each of the factors in the product set $\textbf{r}\lp \lb 0, 1 \rp^{2d-1} \rp$ covers the real line exactly $v_1$ times. Consequently, this sum over all $v_2, \ldots, v_{2d} \in \Z$ equals 
 \begin{align}\label{conclusion2}
\sum_{v_i \in \Z \atop (i\geqslant 2)} \int_{\lb 0, 1\rp^{2d-1}} f\lp K a_y m(A) u_{\textbf{t},s} \textbf{v} \rp \hsp \dd \textbf{t} \hsp \dd s = \int_{\R^{2d-1}} f(Ka_y m(A) (v_1, r_2, \ldots, r_{2d})^\intercal ) \hsp \dd \textbf{r}.
\end{align} 
whenever $v_1 \neq 0$.
\par By (\ref{conclusion1}) and (\ref{conclusion2}) we now have 
 \begin{align*}
\sum_{\textbf{v} \in \Z^{2d} \setminus \lbr \0 \rbr} \int_{\lb 0, 1 \rp^{2d-1}} f\lp K a_y m (A) u_{\textbf{t},s} \textbf{v} \rp \hsp \dd \textbf{t} \hsp \dd s &= \sum_{\widetilde{\textbf{v}} \in \Z^{2d-2} \setminus \lbr \0 \rbr} F_{K,y} \lp A \widetilde{\textbf{v}} \rp  + \sum_{v_{2d} \in \Z \setminus \lbr 0 \rbr} f \lp v_{2d} \textbf{x} \rp \\ &\quad \quad + \sum_{v_1 \neq 0} \int_{\R^{2d-1}} f \lp K a_y m(A) \lp v_1, r_2, \ldots, r_{2d} \rp^\intercal \rp \hsp \dd \textbf{r}.
\end{align*} 
When we integrate this over $Y'$, Siegel's mean value theorem shows that the integral of the first sum is
 \begin{align*}
\int_{\R^{2d-2}} F_{K,y}(\textbf{r}) \hsp \dd \textbf{r} &= y \int_{\R^{2d-1}} f \lp K \lp 0, \textbf{r} \rp^\intercal \rp \hsp \dd \textbf{r}.
\end{align*} 
In the integral of the last sum, we note that the $\SP(2d-2, \R)$-invariance of Lebesgue measure on $\R^{2d-2}$ allows us to drop the $m(A)$ at the cost of a factor of $\mathrm{Vol} \lp Y'\rp = 1$. Hence the integral of the last sum is
 \begin{align*}
\sum_{v_1 \neq 0} \int_{\R^{2d-1}} f \lp K a_y \lp v_1, \textbf{r} \rp^\intercal \rp \hsp \dd \textbf{r} = y \sum_{v_1 \neq 0} \int_{\R^{2d-1}} f \lp K \lp yv_1, \textbf{r} \rp^\intercal \rp \hsp \dd \textbf{r}.
\end{align*} 
Putting everything together, we obtain the claimed formula for $P_f$. \par 
To prove the last claim, we note that the first sum in the formula vanishes for $\Vert \textbf{x} \Vert$ large enough since $f$ has compact support, and that the second sum is a (one-dimensional) Riemann sum for $\int f \, \dd \textbf{r}$ with step length $1/ \Vert \textbf{x} \Vert$. Hence, it suffices to show the claim when $\textbf{x}$ lies in some ball centered at the origin, say $\Vert \textbf{x} \Vert \leqslant R$. To this end, note first that due to the compact support of $f$, when $\Vert \textbf{x} \Vert$ is small enough, the Riemann sum equals
 \begin{align*}
\lev \textbf{x} \rev^{-1} \int_{\R^{2d-1}} f \lp K \lp 0, \textbf{r} \rp^\intercal \rp \hsp \dd \textbf{r} \leqslant \frac{C_0}{\lev \textbf{x} \rev}
\end{align*} 
for some $C_0 > 0$. Hence, we may certainly find $C_1 > 0$ such that
 \begin{align*}
\left| \int_{\R^{2d}} f \hsp \dd \textbf{r} - \Vert \textbf{x} \Vert^{-1} \sum_{n \in \Z} \int_{\R^{2d-1}} f \lp K_\textbf{x} \lp \Vert \textbf{x} \Vert^{-1} n, \textbf{r} \rp^\intercal \rp \hsp \dd \textbf{r} \right| \leqslant \frac{C_0}{\lev \textbf{x} \rev} + C_1 
\end{align*} 
for all $\Vert \textbf{x} \Vert \leqslant R$, $\textbf{x} \neq 0$. Next, we observe that if $M$ and $m$ denote the maximal and minimal lengths, respectively, of any vector in the support of $f$, we have
 \begin{align*}
\sum_{n \neq 0} f(n \textbf{x}) \leqslant \lev f \rev_\infty \sum_{n \neq 0} \mathbbm{1} \lp | n | \in \lb \Vert \textbf{x} \Vert^{-1}m, \Vert \textbf{x} \Vert^{-1}M \rb \rp = O\lp \Vert \textbf{x} \Vert^{-1} \rp. 
\end{align*} 
Combining these results and observing that $C_1 = O \lp \Vert \textbf{x} \Vert^{-1} \rp$ for $\Vert \textbf{x} \Vert \leqslant R$, we obtain the claim in this case.
\end{proof}
We now prove (\ref{aux01}) by verifying that 
 \begin{align*}
\lim_{N \rightarrow \infty} \big\| G_N \big\|_2^2 = \lim_{N \rightarrow \infty} \sum_{\pm s = 0}^{N-1} \lp 1 - \frac{|s|}{N} \rp \lp \int_{\Y} \widehat{\chi}_2 \cdot \lp \widehat{\chi}_2 \circ b^s \rp - \mu \lp \widehat{\chi}_2 \rp^2 \hsp \dd \mu \rp < \infty,
\end{align*} 
cf. the proof of (\ref{smoothf-finitevariance}). This, in turn, will follow from the dominated convergence theorem if we can show that with $f = \chi_2$ and $g = f \circ b^s$,
 \begin{align}\label{dccomics}
\sum_{\pm s = 0}^\infty \left| \int_{\Y} \widehat{f} \widehat{g} - \mathrm{Leb}(f)^2 \hsp \dd \mu  \right| < \infty.
\end{align} 
We begin by noting that if $\textbf{x} = j^{-1} b^{-s} (\textbf{x}_1, \textbf{x}_2 ) \in \mathrm{supp} (g_j) = j^{-1}b^{-s} \Omega_2$, one has
 \begin{align*}
\Vert \textbf{x} \Vert^2 = j^{-2} \lp 2^{-2s} \Vert \textbf{x}_1 \Vert^2 + 2^{2s} \Vert \textbf{x}_2 \Vert^2 \rp \gg j^{-2} (2^{-2s} + 2^{2s}) \gg j^{-2} 2^{2|s|},  
\end{align*} 
and hence $\min \lbr \Vert \textbf{x} \Vert : \textbf{x} \in j^{-1} b^{-s} \Omega_2 \rbr \gg j^{-1}2^{|s|}$. It now follows from Theorem \ref{kelmerthm}, Proposition \ref{riemanntotherescue}, and this estimate that 
 \begin{align*}
\int_\Y \widehat{f} \widehat{g} - \mathrm{Leb} (f)^2 \hsp \dd \mu &= \frac{1}{\zeta(2d)} \sum_{j \geqslant 1} \int_{\R^{2d}} P_{f} (\textbf{x}) g(j\textbf{x}) \, \, \dd \textbf{x} - \mathrm{Leb}(f)^2 \\
&= \frac{1}{\zeta(2d)} \sum_{j \geqslant 1} \lp \int_{\R^{2d}} P_f(\textbf{x})g(j\textbf{x}) \hsp \dd \textbf{x} - \mathrm{Leb}(g_j) \mathrm{Leb}(f) \rp \\
&= \frac{1}{\zeta(2d)} \sum_{j \geqslant 1} \int_{j^{-1} b^{-s} \Omega_2} \lp P_f(\textbf{x}) - \mathrm{Leb}(f) \rp \hsp \dd \textbf{x}  \\
&\ll \sum_{j \geqslant 1} \int_{j^{-1} b^{-s} \Omega_2} \Vert \textbf{x} \Vert^{-1} \hsp \dd \textbf{x} \\[+0.5em]
&\ll \sum_{j \geqslant 1} \vol (j^{-1} b^{-s} \Omega_2) j 2^{-|s|} \\[+0.5em]
&= \zeta(2d-1) \vol (\Omega_2) 2^{-|s|},
\end{align*} 
where we recall that $g_j$ denotes the composition of $g$ with multiplication by $j$. This proves (\ref{dccomics}) and hence (\ref{aux01}).\par 
We now prove (\ref{aux02}). We first need the following lemma.
\begin{lemma}\label{BGeq415}
We have $ \lev \widehat{f_\varepsilon} \lp 1 - \eta_L \rp \rev_2 \ll_{f_\varepsilon} L^{-(d - 2)/2}$.
\end{lemma}
\begin{proof}
We follow the proof of \cite[Eq. (4.15)]{BG2}. Proposition \ref{schmidt-alpha} shows that
 \begin{align*}
\lev \widehat{f_\varepsilon} \lp 1 - \eta_L \rp \rev_2^2 = \int_{\Y} \widehat{f_\varepsilon}^2 \lp 1 - \eta_L \rp^2 \hsp \dd \mu \ll_{\mathrm{supp}\lp f_\varepsilon \rp} \lev f_\varepsilon \rev_\infty^2 \int_{\Y} \mathbbm{1}\lp \alpha \geqslant c^{-1} L \rp \alpha^2 \hsp \dd \mu, 
\end{align*} 
since, by construction, $1 - \eta_L = 0$ whenever $\alpha < c^{-1} L$. If we now have real numbers $v$ and $w$ satisfying $1/v + 1/w = 1$, Hölder's inequality and Theorem \ref{aaalphaint} show that the right-hand side is at most 
 \begin{align*}
 \lev f_\varepsilon \rev_\infty^2 \cdot \lev \alpha^2 \rev_v \cdot \mu \lp \lbr \alpha \geqslant c^{-1} L \rbr \rp^{1/w} \ll_{f_\varepsilon, \, p}  \lev \alpha \rev_{2v}^2 \lp c^{-1} L \rp^{-p/w}
\end{align*} 
for any $p \leqslant d$. In order to maximize $p/w$, we take $p = d$ and $v = d/2$ where the last choice ensures that $\lev \alpha \rev_{2v} < \infty$. All in all, this yields
 \begin{align*}
\lev \widehat{f_\varepsilon} \lp 1 - \eta_L \rp \rev_2^2 \ll_{f_\varepsilon} L^{-(d-2)},
\end{align*} 
and the lemma follows.
\end{proof}
Consider
 \begin{align*}
G_N^{(\epsilon, L)} - G_N = \frac{1}{\sqrt{N}} \lp \sum_{m = 0}^{N-1}   \lp \widehat{f}_\epsilon \eta_L - \widehat{\chi}_2 \rp \circ b^m  - N \cdot \mu \lp \widehat{f}_\epsilon \eta_L - \widehat{\chi}_2 \rp \rp.
\end{align*} 
Applying the triangle inequality and reusing the estimate of $\lv \widehat{f}_\epsilon \eta_L - \widehat{\chi}_2 \rv_1$ from the proof of Lemma \ref{smoothapprox}, we find that
 \begin{align*}
\lv G_N^{(\epsilon, L)} - G_N \rv_2 &\leqslant \sqrt{N} \lp \lv  \widehat{f}_\epsilon \eta_L - \widehat{\chi}_2 \rv_2 +  \lv \widehat{f}_\epsilon \eta_L - \widehat{\chi}_2 \rv_1 \rp 
\\ &\leqslant  \sqrt{N} \lp \lv  \widehat{f}_\epsilon \lp 1 - \eta_L \rp \rv_2 + \lv \widehat{f}_\epsilon - \widehat{\chi}_2 \rv_2   +  L^{-d/2} + \epsilon \rp.
\end{align*} 
Applying Lemma \ref{BGeq415} and Lemma \ref{siegelestnorm} to the right-hand side, we find that
 \begin{align*}
\lv G_N^{(\epsilon, L)} - G_N \rv_2 &\ll_{f_\varepsilon} \sqrt{N} \lp L^{-(d - 2)/2} + \lv f_\epsilon - \chi_2 \rv_1 + \lv f_\epsilon - \chi_2 \rv_2 + L^{-d/2} + \epsilon \rp
\\ 
&\ll \sqrt{N} \lp L^{-(d-2)/2} + \sqrt{\epsilon} \rp,
\end{align*} 
where we used (\ref{feps01}) and (\ref{feps02}) in the last step. We conclude that (\ref{aux02}) holds provided that $L$ and $\epsilon$ satisfy the conditions
 \begin{align}\label{conditionsarechanging}
\epsilon = o \lp N^{-1} \rp, \quad \quad N = o \lp L^{d-2} \rp.
\end{align}  
Note that these conditions imply that (\ref{conditionsonepsandl}) is satisfied. We postpone the matter of choosing $L$ and $\varepsilon$ since other conditions will have to be taken into account, too, in order to ensure (\ref{GN-cumtozero}). In conclusion, aside from ensuring (\ref{conditionsarechanging}), we have proved (\ref{aux02}). Along with (\ref{aux01}), this proves (\ref{GN-finitevariance}). 
\subsubsection{Initial Estimates of the Cumulants}
We now begin the proof of (\ref{GN-cumtozero}). The proof of Theorem \ref{cltforsmoothfcts} shows that with $\phi = \widehat{f_\varepsilon} \eta_L$, we have the estimate
 \begin{align*}
\nonumber \cum_r \lp G_N^{(\epsilon, L)} \rp &\ll_{q,r}  N^{1-r/2} \beta_{r }^{r-1} \lv \phi \rv_\infty^r + \max \lbr 1, S_{q } \lp \phi \rp^r \rbr N^{r/2} \sum_{j = 0}^r e^{- \lp \delta' \beta_{j+1} - q \xi \alpha_j r \rp }.
\end{align*} 
However, we will need a refinement of the estimate (\ref{cumest}) in order to improve the factor $\lv \phi \rv_\infty^r$ in the estimate above. To do so, we adapt the result \cite[Eq. (5.17)]{BG2} to the current situation.
\begin{lemma}\label{degnenihegnet}
For $n = 1, 2, \ldots, r$, let $\psi_{n} = \phi \circ b^n - \mu \lp \phi \rp$. We then have
 \begin{align*}
\cum_r \lp \psi_{m_1}, \ldots, \psi_{m_r} \rp \ll_{r, d, \mathrm{supp} \, f_\varepsilon} \lv \phi \rv_\infty^{(r-d)^+},
\end{align*} 
where $(r-d)^+ = \max \lbr 0, r-d \rbr$. 
\end{lemma}
\begin{proof}
By definition of the cumulant $\cum_r \lp \psi_{m_1}, \ldots, \psi_{m_r} \rp$ it is enough to show that for any of the possible values of $n$,
 \begin{align}\label{genholder01}
\int_{\Y}  \left|  \psi_{m_1} \cdots\psi_{m_n} \right| \hsp \dd \mu \ll_{n, d, \mathrm{supp} \, f_\varepsilon} \lv \phi \rv_\infty^{(n-d)^+}.
\end{align} 
Suppose first that $n \leqslant d$. Applying the generalized Hölder inequality to the $d$ functions $\psi_{m_1}, \ldots, \psi_{m_n}, 1, \ldots, 1$, we find that
 \begin{align}\label{genholder02}
\int_{\Y} \left| \psi_{m_1} \cdots \psi_{m_n} \right| \hsp \dd \mu \leqslant \lv \psi_{m_1} \rv_d \cdots \lv \psi_{m_n} \rv_d \ll_{d, \mathrm{supp} \, f_\varepsilon} 1.
\end{align} 
Indeed, for any $m$ we have
 \begin{align*}
\lv \psi_m \rv_d \leqslant \mu (\phi) + \lv \widehat{f_\varepsilon} \rv_d \ll_{d, \mathrm{supp} \, f_\varepsilon } \mu(\phi) + \lv \alpha \rv_d
\end{align*} 
by the $G$-invariance of $\mu$, Proposition \ref{schmidt-alpha}, and 
 \begin{align*}
\mu (\phi) \leqslant \lv f_\varepsilon \rv_1 \leqslant \lev f_\varepsilon - \chi_2 \rv_1 +  \lv \chi_2 \rv_1 \ll \varepsilon + \mathrm{Vol}(\Omega_2) 
\end{align*} 
by (\ref{feps01}). In combination with Theorem \ref{aaalphaint}, this proves (\ref{genholder02}). \par 
Next, suppose that $n > d$ and consider the $d+1$ functions $\psi_{m_1}, \, \ldots, \psi_{m_d}, \, \psi_{m_{d+1}} \cdots \psi_{m_n}$. Once again applying the generalized Hölder inequality, we obtain
 \begin{align}\label{genholder03}
\nonumber \int_{\Y} \left| \psi_{m_1} \cdots \psi_{m_n} \right| \hsp \dd \mu &\leqslant \lv \psi_{m_{d+1}} \cdots \psi_{m_n} \rv_\infty \int_{\Y} \left| \psi_{m_1} \cdots \psi_{m_d} \right| \hsp \dd \mu \\ 
\nonumber &\leqslant 2^{n-d} \lv \phi \rv_\infty^{n-d} \lv \psi_{m_1} \rv_d \cdots \lv \psi_{m_d} \rv_d \\[+0.4em]
\nonumber &\ll_{d, \mathrm{supp} \, f_\varepsilon} 2^{n-d} \lv \phi \rv_\infty^{n-d} \\[+0.4em]
&\ll_{n, d} \lv \phi \rv_\infty^{n-d},
\end{align} 
where we reused the above estimate of $\lv \psi_m \rv_d$. This proves (\ref{genholder01}) and hence the lemma.
\end{proof}
\textsc{Remark.} We stress that the implied constant in Lemma \ref{degnenihegnet} is allowed to depend on the support of $f_\varepsilon$, but does not depend on $L$.\\\par 
Now, since $\eta_L \leqslant 1$ is supported on the set $\lbr \Lambda \in \Y : \alpha (\Lambda) \leqslant c L \rbr$,
 \begin{align*}
\lv \phi \rv_\infty &= \lv \widehat{f_\varepsilon} \eta_L \rv_\infty 
\leqslant \sup_{\Lambda \in \Y} \lbr \left| \widehat{f}_\varepsilon (\Lambda) \right| : \alpha (\Lambda) \leqslant c L \rbr \\ 
&\ll_{\text{supp } f_\epsilon} \lv f_\varepsilon \rv_\infty \sup_{\Lambda \in \Y} \lbr \left| \alpha ( \Lambda ) \right| : \alpha (\Lambda ) \leqslant cL\rbr \ll_{f_\varepsilon, \, c} L,
\end{align*} 
by Proposition \ref{schmidt-alpha}. From this and Lemma \ref{degnenihegnet} we obtain
 \begin{align}\label{genholdertotherescue}
\cum_r \lp G_N^{(\epsilon, L)} \rp &\ll_{q,c, f_\epsilon}  2^{r} P(r) N^{1-r/2} \beta_{r }^{r-1} L^{(r-d)^+} + \max \lbr 1, S_{q } \lp \phi \rp^r \rbr N^{r/2} \sum_{j = 0}^r e^{- \lp \delta' \beta_{j+1} - q \xi \alpha_j r \rp },
\end{align} 
where $2^r P(r)$ denotes (an upper bound for) the function of $r$ appearing in the implied constant of Lemma \ref{degnenihegnet}.
\subsubsection{Estimating the Sobolev Norm $S_q (\phi)$}
We now proceed by estimating $S_{q} \lp \phi \rp$. First of all, we note that since $\phi = \widehat{f}_\epsilon \eta_L$ and $\eta_L (\Lambda)$ vanishes if $\alpha (\Lambda) > cL$, all suprema of $\phi$ and its Lie derivatives may be taken over the set $ \alpha^{-1} \bigl( \lb 0, c L \rb \bigr) \subset \Y$. If we first consider the $L^\infty$-norm, we therefore see that
 \begin{align}\label{infty-estimate}
\nonumber \lv \widehat{f}_\varepsilon \eta_L \rv_\infty &= \sup \lbr \left| \widehat{f}_\epsilon (\Lambda) \eta_L (\Lambda) \right| : \Lambda \in \Y, \, \alpha (\Lambda) \leqslant cL \rbr 
\\
\nonumber &\leqslant \lp \sup_{\alpha \leqslant cL} \widehat{f}_\varepsilon (\Lambda) \rp \lp \sup_{\alpha \leqslant cL} \eta_L (\Lambda) \rp \\
\nonumber &\ll_{c, \text{supp } f_\epsilon} \lv f_\varepsilon \rv_\infty \sup_{\alpha \leqslant cL} \alpha (\Lambda)
\\
&\ll_c L. 
\end{align} 
\par Next, let us consider the $C^q$-norm. If $Z \in \mathcal{U}(\mathfrak{g})$ is any monomial of degree at most $q$, then as before, since $\mathcal{D}_Z$ is a derivation on $C_c^\infty ( \Y )$, we find that 
 \begin{align}\label{derivation-sum}
\mathcal{D}_Z \lp \widehat{f_\varepsilon} \eta_L \rp = \sum_{Z', Z''} \mathcal{D}_{Z'} \lp \widehat{f_\varepsilon} \rp \mathcal{D}_{Z''} \lp \eta_L \rp
\end{align} 
where the sum extends over all monomials $Z', Z'' \in \mathcal{U}(\mathfrak{g})$ of degree at most $q$ and satisfying $\deg Z' + \deg Z'' = \deg Z$. We note that, although $\widehat{f_\varepsilon}$ is not a compactly supported function on $\Y$, the symbol $\mathcal{D}_{Z'} \lp \widehat{f_\varepsilon} \rp$ still carries meaning. Indeed, if we write $\Lambda = g \Gamma$, then for $Z'$ of degree $1$,
 \begin{align*}
\lp \mathcal{D}_{Z'} \widehat{f_\varepsilon} \rp \lp \Lambda \rp 
&= \lim_{t \rightarrow 0} \hsp t^{-1} \lp \widehat{f_\varepsilon}  \lp e^{tZ'} g \Gamma \rp - \widehat{f_\varepsilon} \lp g \Gamma \rp \rp = \lim_{t \rightarrow 0} \hsp \sumprime_{\textbf{v} \in g \Z^{2d}} t^{-1} \lp f_\varepsilon \lp e^{tZ'}\textbf{v} \rp - f_\varepsilon (\textbf{v}) \rp.
\end{align*} 
Since the sum above is finite, we obtain
 \begin{align}\label{siegelcommuteswithdiff}
\lp \mathcal{D}_{Z'} \widehat{f_\varepsilon} \rp \lp \Lambda \rp = \sumprime_{\textbf{v} \in g \Z^{2d}} \lp \mathcal{D}_{Z'} f_\varepsilon \rp \lp \textbf{v} \rp = \lp \widehat{\mathcal{D}_{Z'} f_\varepsilon} \rp \lp \Lambda \rp,
\end{align} 
where $\mathcal{D}_{Z'}$ now also represents the differential operator on $C_c^\infty \lp \R^{2d} \rp$ given by
 \begin{align*}
\lp \mathcal{D}_{Z'} f \rp ( \textbf{x} ) = \left. \frac{\partial}{\partial t} \right|_{t = 0} \lp f \lp e^{tZ'} \textbf{x} \rp - f (\textbf{x}) \rp, \quad \quad f \in C_c^\infty \lp \R^{2d} \rp, \quad \quad \textbf{x} \in \R^{2d}.
\end{align*} 
The case of a monomial $Z'$ of arbitrary degree follows easily by induction. It now follows from (\ref{siegelcommuteswithdiff}), Proposition \ref{schmidt-alpha}, and the properties of $f_\epsilon$ that we have
 \begin{align}\label{kindofblue}
\sup_{\alpha \leqslant cL} \left| \lp \mathcal{D}_{Z'} \widehat{f}_\epsilon \rp (\Lambda) \right| \ll_{\text{supp } f_\epsilon} \lv \mathcal{D}_{Z'} f_\varepsilon \rv_\infty \sup_{\alpha \leqslant cL} \alpha ( \Lambda) \ll_c \epsilon^{-q} L,
\end{align} 
since $\deg Z' \leqslant \deg Z \leqslant q$. Moreover, by (\ref{etadef}) and the $G$-invariance of $\mu$, we have
 \begin{align*}
\mathcal{D}_{Z''} \lp \eta_L \rp = \mathcal{D}_{Z''} \lp \Xi * \mathbbm{1}_{\lbr \alpha \leqslant L \rbr} \rp = \lp \mathcal{D}_{Z''} \Xi  \rp * \mathbbm{1}_{\lbr \alpha \leqslant L \rbr},
\end{align*} 
and hence
 \begin{align*}
\lv \mathcal{D}_{Z''} \lp \eta_L \rp \rv_\infty \leqslant \lv \mathcal{D}_{Z''} \Xi \rv_1 \ll_q 1.
\end{align*} 
From this, (\ref{kindofblue}), and (\ref{derivation-sum}) we now conclude that
 \begin{align*}
\lv \mathcal{D}_Z \lp \phi \rp \rv_{\infty} = \sup_{\alpha \leqslant cL} \left| \mathcal{D}_{Z} \lp \widehat{f}_\epsilon \eta_L \rp \right| \ll_{c,q, \text{supp } f_\epsilon} \epsilon^{-q}L,
\end{align*} 
and by taking the supremum of the left-hand side over all monomials $Z$ of degree at most $q$, we establish the bound
 \begin{align*}
\lv \phi \rv_{C^q} \ll_{c,q,\text{supp } f_\epsilon} \varepsilon^{-q} L.
\end{align*} 
Along with (\ref{infty-estimate}), this proves that
 \begin{align}\label{sq-est-lava}
S_q \lp \widehat{f_\epsilon} \eta_L \rp = S_q \lp \phi \rp \ll_{c,q} \epsilon^{-q} L.
\end{align} 
We note that the implied constant is indeed independent of $f_\epsilon$ since the family $\lbr \mathrm{supp} \, f_\epsilon \rbr_\epsilon$ is uniformly bounded.
\subsubsection{Optimizing the Parameters $\epsilon$ and $L$}
By (\ref{genholdertotherescue}) and the estimate (\ref{sq-est-lava}), we obtain 
 \begin{align*}
\cum_r \lp G_N^{(\epsilon, L)} \rp &\ll_{q,c, f_\varepsilon} 2^r P(r) N^{1-r/2} \beta_{r }^{r-1} L^{(r-d)^+} + \varepsilon^{-qr} L^r N^{r/2} \sum_{j = 0}^r e^{- \lp \delta' \beta_{j+1} - q \xi \alpha_j r \rp } \\
&\ll  2^r P(r) N^{1-r/2} \gamma^{r-1} L^{(r-d)^+} + \varepsilon^{-qr} L^r N^{r/2} e^{-\delta ' \gamma}
\end{align*} 
where the last inequality follows by choosing the sequence $\lbr \beta_i \rbr$ as in (\ref{choiceofbeta}). Since $r$ is arbitrary and thus constant, all that remains in order to prove Theorem \ref{maintheorem} for $T = 2^N$ is therefore to choose the parameters $\varepsilon$ and $L$ so that (\ref{conditionsarechanging}) is satisfied and so that 
 \begin{align}
N^{1-r/2} \gamma^{r-1} L^{(r-d)^+} &\longrightarrow 0, \label{conv01}\\
\varepsilon^{-qr} L^r N^{r/2} e^{-\delta ' \gamma} &\longrightarrow 0. \label{conv02}
\end{align} 
As in \textsc{Section 4}, we will also here take $\gamma$ to be some multiple $C_\gamma \log N$ of $\log N$. We can then assume that $r-d > 0$, since otherwise (\ref{conv01}) follows immediately as $r \geqslant 3$. Then, if we take $L = N^j$ for some real number $j$ to be determined later, we see that (\ref{conv01}) is satisfied provided that $1 - r/2 + j (r-d) < 0$. That is, $j$ should satisfy
 \begin{align*}
j < \frac{r-2}{2(r - d )}.
\end{align*} 
In order for (\ref{conditionsarechanging}) to be satisfied, we also need that $( d-2 ) j > 1$. In summary, we can find such a $j$ if and only if
 \begin{align}\label{apr22-ineq}
\frac{1}{d-2} < \frac{r-2}{2(r-d)}.
\end{align} 
This is equivalent to $r(d-4) + 4 > 0$, which is true since $d \geqslant 4$. Hence, we may find a suitable $j$ such that (\ref{conv01}) is satisfied with $L = N^j$. \par 
It remains to choose the constant $C_\gamma = \gamma / \log N$ and $\varepsilon$. Taking $\varepsilon = N^{-3/2}$, we see that (\ref{conv02}) is satisfied if 
 \begin{align*}
r \lp \frac{3q + 1}{2} + j \rp - \delta' C_\gamma < 0,
\end{align*} 
which is true if 
 \begin{align*}
C_\gamma > \frac{r}{\delta'} \lp \frac{3q + 1}{2} + j \rp.
\end{align*} 
Obviously such a choice of $C_\gamma$ is possible. Therefore, with the given choices of $L$ and $\epsilon$, also (\ref{conv02}) is satisfied. Since (\ref{conv01}) and (\ref{conv02}) are now satisfied, we conclude that (\ref{GN-cumtozero}) is satisfied for all $r \geqslant 3$. Together with Lemma \ref{smoothapprox} (in particular (\ref{hbo-2})), this now proves Theorem \ref{maintheorem} in the case $T = 2^N$.
\subsection{Reduction of Theorem \ref{maintheorem} to the Special Case $T = 2^N$}
We conclude by showing that, in fact, the special case $T = 2^N$ of Theorem \ref{maintheorem} implies the theorem in its full generality. For the real parameter $T$, we let $N = N(T) = \lfloor \log_2 T \rfloor$ so that
 \begin{align}\label{TvsNT}
\frac{T}{2} < 2^{N} \leqslant T < 2^{N + 1}.
\end{align} 
If $\Lambda \in \Y$ denotes any symplectic lattice, we then have
 \begin{align*}
\# \lp \Lambda \cap \Omega_T \rp - \vol \lp \Omega_T \rp &= \# \lp \Lambda \cap \Omega_{2^N} \rp - \vol \lp \Omega_{2^N} \rp + \# \lp \Lambda \cap \lp \Omega_T \setminus \Omega_{2^N} \rp \rp - \vol \lp \Omega_T \setminus \Omega_{2^N} \rp,
\end{align*} 
and hence
 \begin{align*}
X_T := \frac{\# \lp \Lambda \cap \Omega_T \rp - \vol \lp \Omega_T \rp}{\vol \lp \Omega_T \rp^{1/2}} 
&= \alpha_T Z_T + \beta_T - \gamma_T,
\end{align*} 
where we write
 \begin{align*}
\alpha_T := \lp \frac{\vol \lp \Omega_{2^N} \rp  }{\vol \lp \Omega_T \rp} \rp^{1/2}, \quad \beta_T &:= \frac{\# \lp \Lambda \cap \lp \Omega_T \setminus \Omega_{2^N} \rp \rp }{\vol \lp \Omega_T \rp^{1/2}}, \quad \gamma_T := \lv \beta_T \rv_1,
\end{align*} 
and 
 \begin{align*}
Z_T := \frac{\# \lp \Lambda \cap \Omega_{2^N} \rp - \vol \lp \Omega_{2^N} \rp}{\vol \lp \Omega_{2^N} \rp^{1/2}}.
\end{align*} 
So far, we know that as $T \longrightarrow \infty$, $Z_T$ converges in distribution to $N\lp 0, \sigma^2 \rp$ for some $\sigma \geqslant 0$. It will therefore follow by standard methods that $X_T \Longrightarrow N\lp 0, \sigma^2 \rp$ if we can show that 
 \begin{align}
\alpha_T &\longrightarrow 1, \label{alphaconvergence}\\
\gamma_T &\longrightarrow 0. \label{gammaconvergence}
\end{align} 
To this end, we now compute the volume of $\Omega_T$. By changing to spherical coordinates, we see that $\Omega_T$ is the set
 \begin{align*}
\lbr \lp r, \phi_1, \ldots, \phi_{d-1}; \, s, \theta_1, \ldots, \theta_{d-1} \rp \in \lp \R_+ \times \lb 0, \pi \rb^{d-2} \times \lb 0, 2 \pi \rp \rp^2 : 1/s \leqslant r \leqslant 2/s, \, 1 \leqslant s < T \rbr.
\end{align*} 
Writing $r^{d-1} F\lp \phi_1, \ldots, \phi_{d-1} \rp \,  \dd \phi_1 \cdots \dd \phi_{d-1}  \dd r$ for the spherical volume element on $\R^d$, we see that with
 \begin{align*}
c_d = \int_0^{2 \pi} \int_0^\pi \cdots \int_0^\pi F\lp \phi_1, \ldots, \phi_{d-1} \rp \hsp \hsp \dd \phi_1 \cdots \dd \phi_{d-1},
\end{align*} 
the volume of $\Omega_T$ is
 \begin{align*}
\vol \lp \Omega_T \rp  = c_d^2 \int_{1}^T \int_{1/s}^{2/s} r^{d-1} s^{d-1} \hsp \dd r \hsp \dd s = \frac{c_d^2 (2^d - 1)}{d} \log T.
\end{align*} 
\par By taking the logarithm in (\ref{TvsNT}) and dividing by $\log T$, we get that
 \begin{align}\label{whatawonderfulworld}
1 - \frac{\log 2}{\log T} < \frac{N \log 2}{\log T}  < 1 + \frac{\log 2}{\log T},
\end{align} 
and hence
 \begin{align*}
 \lp 1 - \frac{\log 2}{\log T} \rp^{1/2} < \alpha_T = \lp \frac{N \log 2}{\log T} \rp^{1/2} < \lp 1 + \frac{\log 2}{\log T} \rp^{1/2},
\end{align*} 
which proves (\ref{alphaconvergence}). \par 
As for $\gamma_T$, we see that (\ref{TvsNT}) implies
 \begin{align*}
0 \leqslant \log T - N \log 2 < \log 2,
\end{align*} 
and it follows that 
 \begin{align*}
\gamma_T = \lp \frac{c_d^2 (2^d - 1)}{d} \rp^{1/2} \frac{\log T - N \log 2}{\lp \log T \rp^{1/2}} \ll_d \frac{\log 2}{\lp \log T \rp^{1/2}}  \longrightarrow 0.
\end{align*} 
This completes the proof of (\ref{gammaconvergence}). Since now both (\ref{alphaconvergence}) and (\ref{gammaconvergence}) hold, the proof of Theorem \ref{maintheorem} is complete.
\begin{appendices}
\section{A Siegel Set for $G / \Gamma$}\label{siegelappendix}
We now give a proof of the explicit Siegel set decomposition expressed in Proposition \ref{siegeldroid}. We follow the arguments in \cite[Chap. V.1]{bekkamayer}.
\begin{lemma}\label{tooktoolong}
There exists a number $m(d) \geqslant 2$ with the property that ${\bf N} = {\bf N}_{m(d)}  {\bf N}_{\Z}$.
\end{lemma}
\begin{proof}
We proceed by induction in $d$. The case $d = 1$ is well-known (see \cite{bekkamayer}) since $\SP(2, \R) = \SL(2, \R)$, and in this case we can take $m(1) = 2$. 
\par We now let $d \geqslant 2$ and assume the claim for $2(d-1) \times 2(d-1)$ matrices in ${\bf N}(d-1)$ and define $m(d)$ recursively. Given any $\lp \begin{smallmatrix}
N & M \\
0 & N^{-\intercal} \end{smallmatrix} \rp \in {\bf N}(d)$, we have to prove the existence of a $d \times d$ unipotent, upper-triangular, integer matrix $S$ and a $d \times d$ integer matrix $T$ such that 
 \begin{align}
\left\Vert NS \right \Vert_\infty &\leqslant m(d), \label{show01}
\\ \lev \lp NS \rp^{-1} \rev_\infty &\leqslant m(d), \label{show02} 
\\ \left \Vert NT + MS^{-\intercal} \right \Vert_\infty &\leqslant m(d), \label{show03} 
\\ ST^\intercal &= TS^\intercal, \label{show04}
\end{align} 
for some $m(d)$ only depending on $d$. \par 
Let us write $N$ in the form
 \begin{align*}
N = \lp \begin{matrix}
A & \textbf{x} \\
0 & 1
\end{matrix} \rp,
\end{align*} 
where $A$ is a $(d-1) \times (d-1)$ unipotent upper-triangular matrix and $\textbf{x} \in \R^{d-1}$. Then we have
 \begin{align*}
N^{-1} = \lp \begin{matrix}
A^{-1} & -A^{-1} \textbf{x} \\
0 & 1
\end{matrix} \rp. 
\end{align*} 
Applying the induction hypothesis to the matrix $\lp \begin{smallmatrix} A & 0 \\ 0 & A^{-\intercal} \end{smallmatrix} \rp$, we can find a unipotent, upper-triangular integer matrix $S_0$ of dimensions $(d-1) \times (d-1)$ such that $\Vert A S_0 \Vert_\infty \leqslant m(d-1)$ and $\Vert (AS_0)^{-1} \Vert_\infty \leqslant m(d-1)$. \par We now augment $S_0$ to a matrix $S$ satisfying (\ref{show01}) and (\ref{show02}) as follows. Write first
 \begin{align*}
S = \lp \begin{matrix}
S_0 & \textbf{s} \\
0 & 1
\end{matrix} \rp
\end{align*} 
where $\textbf{s} \in \Z^{d-1}$. Then one has
 \begin{align*}
NS &= \lp \begin{matrix}
AS_0 & A\textbf{s} + \textbf{x} \\
0 & 1
\end{matrix} \rp, \\ (NS)^{-1} &= \lp \begin{matrix}
( AS_0  )^{-1} & -( AS_0  )^{-1}(A \textbf{s} + \textbf{x}) \\
0 & 1
\end{matrix} \rp.
\end{align*} 
It is now possible to choose the vector $\textbf{s}$ in such a way that $\Vert A \textbf{s} + \textbf{x} \Vert_\infty \leqslant 2$. Indeed, this claim is equivalent to the existence of a point of the lattice $A \Z^{d-1}$ in the set $-\textbf{x} + \lb -2,2 \rb^{d-1}$. This claim, on the other hand, is true since surely $-\textbf{x} + \lb -2,2 \rb^{d-1}$ contains one of the cubes 
 \begin{align*}
\lbr \textbf{v} + \lb 0,1 \rb^{d-1} : \textbf{v} \in \Z^{d-1} \rbr,
\end{align*} 
each of which necessarily contains a lattice point from $A \Z^{d-1}$, as may be seen by solving the system of inequalities
 \begin{align*}
A \textbf{s} = \lp \begin{matrix}
1 & a_{12} & a_{13} & \cdots & a_{1,d-1}\\
0 & 1 & a_{23} & \cdots & a_{2,d-1}\\
0 & 0 & 1 & \cdots & a_{3,d-1} \\
\vdots &\vdots &\vdots &\ddots &\vdots\\
0 & 0 & 0 & \cdots & 1
\end{matrix} \rp \lp \begin{matrix}
s_1 \\ s_2 \\ s_3 \\ \vdots \\ s_{d-1}
\end{matrix} \rp \in \textbf{v} + \lb 0, 1 \rb^{d-1}
\end{align*} 
by starting with $s_{d-1}$. It follows that we now have $\Vert A S_0 \Vert_\infty \leqslant m(d-1)$, $\Vert (AS_0 )^{-1} \Vert_\infty \leqslant m(d-1)$, and $\Vert A \textbf{s} + \textbf{x} \Vert_\infty \leqslant 2$. However, these inequalities imply that additionally
 \begin{align*}
\lev( AS_0  )^{-1}(A \textbf{s} + \textbf{x})\rev_\infty \leqslant 2(d-1) m(d-1).
\end{align*} 
We therefore obtain 
 \begin{align}\label{partial01}
\lev NS \rev_\infty \leqslant \max \lbr m(d-1), 2 \rbr = m(d-1) < 2(d-1) m(d-1),
\end{align} 
since $m(d-1) \geqslant 2$ by assumption, and 
 \begin{align}\label{partial02}
\lev (NS)^{-1} \rev_\infty \leqslant 2(d-1) m(d-1).
\end{align} 
\par
Let $u(d) = 2(d-1) m(d-1)$. With $S$ as above, we now let $T'$ be any integer matrix satisfying $ S (T')^\intercal = T' S^\intercal$ so that 
 \begin{align*}
\lp \begin{matrix}
N & M \\
0 & N^{-\intercal} \end{matrix} \rp 
\lp \begin{matrix}
S & T' \\
0 & S^{-\intercal} \end{matrix} \rp 
=
\lp \begin{matrix}
N' & M' \\
0 & (N')^{-\intercal} \end{matrix} \rp
\end{align*} 
with $\Vert N' \Vert_\infty, \, \Vert (N')^{-\intercal} \Vert_\infty \leqslant u(d)$. Thus, if we can find a symmetric integer matrix $T''$ such that $\Vert N' T'' + M' \Vert_\infty$ is bounded, we will be done. Indeed, then we may take $m(d)$ as an appropriate multiple of $u(d)$ and define $T$ by the relation 
 \begin{align}\label{matrices}
\lp \begin{matrix}
S & T \\
0 & S^{-\intercal} \end{matrix} \rp = 
\lp \begin{matrix}
S & T' \\
0 & S^{-\intercal} \end{matrix} \rp 
\lp \begin{matrix}
I & T'' \\
0 & I \end{matrix} \rp.
\end{align} 
\par In order to find such a matrix $T''$, we note that the condition $N' (M')^\intercal = M' (N')^\intercal$ implies that $(N')^{-1} M'$ is symmetric. Indeed, 
 \begin{align*}
\lp (N')^{-1} M' \rp^\intercal = (M')^\intercal (N')^{- \intercal} = (N')^{-1} M' (N')^{\intercal} (N')^{-\intercal} = (N')^{-1} M'.
\end{align*} 
Therefore we obtain a symmetric integer matrix by letting
 \begin{align*}
T'' = \lp t_{ij} \rp_{i,j = 1}^d, \quad \quad t_{ij} = \left\lfloor \lp -(N')^{-1} M' \rp_{ij} \right\rfloor.
\end{align*} 
For a suitable $d \times d$ matrix $U$ with $\Vert U \Vert_\infty \leqslant 1$, we now have $T'' = -(N')^{-1} M' + U$. This implies that
 \begin{align*}
\lev N' T'' + M' \rev_\infty = \lev N' U \rev_\infty \leqslant d \cdot u(d).  
\end{align*} 
Hence we can take $m(d) = d \cdot u(d)$. By (\ref{partial01}) and (\ref{partial02}), the conditions (\ref{show01})-(\ref{show04}) are satisfied with $S$ and $T$ given by (\ref{matrices}).
\end{proof}
For $M$ a $(d-1) \times (d-1)$ matrix, $a \in \R$, and $\textbf{x}, \textbf{y} \in \R^{d-1}$, let us introduce the $d \times d$ matrix
 \begin{align*}
\mathcal{M}(a, \textbf{x}, \textbf{y}, M ) := \lp \begin{matrix}
a & \textbf{x}^\intercal\\
\textbf{y} & M
\end{matrix} \rp.
\end{align*} 
To complete the induction step and obtain a Siegel set for $\SP (2d, \R) / \SP (2d, \Z)$, we will rely on the following lemma which allows us to project and lift matrices to and from $\SP(2d-2, \R)$.
\begin{lemma}\label{symplecticlemma} For $i = 1,2,3$, let $a_i \in \R$ and let ${\bf x}_i \in \R^{d-1}$. Moreover, for $i = 1, 2, 3, 4$, let $M_i \in \mathrm{GL}(d-1, \R)$. \\ \textit{i)} Suppose that
 \begin{align*}
\lp 
\begin{matrix}
\mathcal{M}(a_1, {\bf x}_1, \0, M_1 ) & \mathcal{M}(a_2, {\bf x}_2, {\bf x}_3, M_2) \\
0 & *
\end{matrix}
\rp \in \SP(2d, \R).
\end{align*} 
Then 
 \begin{align*}
\lp \begin{matrix}
M_1 & M_2 \\
0 & M_1^{-\intercal}
\end{matrix} \rp \in \SP (2d-2, \R).
\end{align*} 
\\ \textit{ii)} Suppose that
 \begin{align*}
\lp \begin{matrix}
M_1 & M_2 \\
M_3 & M_4
\end{matrix} \rp \in \SP (2d-2, \R).
\end{align*} 
Then
 \begin{align*}
\lp 
\begin{matrix}
\mathcal{M}(1,\0, \0, M_1 ) & \mathcal{M}(0, \0, \0, M_2 ) \\
\mathcal{M}(0, \0, \0, M_3 ) & \mathcal{M}(1, \0, \0, M_4 ) 
\end{matrix}
\rp \in \SP(2d, \R).
\end{align*} 
\end{lemma}
\begin{proof} We first prove claim \textit{i)}. The assumption implies that
 \begin{align*}
\mathcal{M}(a_1, {\bf x}_1, \0, M_1 ) \mathcal{M}(a_2, {\bf x}_2, {\bf x}_3, M_2)^\intercal = \mathcal{M}(a_2, {\bf x}_2, {\bf x}_3, M_2) \mathcal{M}(a_1, {\bf x}_1, \0, M_1 )^\intercal.
\end{align*} 
By computing the transposes and matrix products on both sides, one immediately obtains
 \begin{align*}
\mathcal{M} \lp *, *, *, M_1 M_2^\intercal \rp = \mathcal{M} \lp *,*,*, M_2 M_1^\intercal \rp,
\end{align*} 
and the claim follows. \par We now prove claim \textit{ii)}. The assumption implies that
 \begin{align}\label{symplecticcondition}
M_1^\intercal M_3 = M_3^\intercal M_1, \quad \quad M_2^\intercal M_4 = M_4^\intercal M_2 , \quad \quad M_1^\intercal M_4 = I + M_3^\intercal M_2.
\end{align} 
Using this, one now computes that
 \begin{align*}
&\lp 
\begin{matrix}
\mathcal{M}(1, \0, \0, M_1 ) & \mathcal{M}(0, \0, \0, M_2 ) \\
\mathcal{M}(0, \0, \0, M_3 ) & \mathcal{M}(1, \0, \0, M_4 ) 
\end{matrix}
\rp^\intercal \lp \begin{matrix}
0 & I \\
-I & 0 
\end{matrix} \rp 
\lp 
\begin{matrix}
\mathcal{M}(1, \0, \0, M_1 ) & \mathcal{M}(0, \0, \0, M_2 ) \\
\mathcal{M}(0, \0, \0, M_3 ) & \mathcal{M}(1, \0, \0, M_4 ) 
\end{matrix}
\rp 
\\[+1em] &\quad \quad = 
\lp 
\begin{matrix}
\mathcal{M}(0, \0, \0, -M_3^\intercal ) & \mathcal{M}(1, \0, \0, M_1^\intercal ) \\
\mathcal{M}(-1, \0, \0, -M_4^\intercal ) & \mathcal{M}(0, \0, \0, M_2^\intercal ) 
\end{matrix}
\rp
\lp 
\begin{matrix}
\mathcal{M}(1, \0, \0, M_1 ) & \mathcal{M}(0, \0, \0, M_2 ) \\
\mathcal{M}(0, \0, \0, M_3 ) & \mathcal{M}(1, \0, \0, M_4 ) 
\end{matrix}
\rp \\[+1em]
&\quad \quad = \lp
\begin{matrix}
\mathcal{M} (0, \0, \0, -M_3^\intercal M_1 + M_1^\intercal M_3 ) & \mathcal{M} (1, \0, \0, -M_3^\intercal M_2 + M_1^\intercal M_4 ) \\
\mathcal{M} (-1, \0, \0, -M_4^\intercal M_1 + M_2^\intercal M_3 ) & \mathcal{M} (0, \0, \0, -M_4^\intercal M_2+  M_2^\intercal M_4 )
\end{matrix} \rp \\[+1em]
&\quad \quad = \lp
\begin{matrix}
0 & \mathcal{M}(1, \0, \0, I) \\
\mathcal{M}(-1, \0, \0, -I) & 0
\end{matrix} \rp = \lp \begin{matrix}
0 &I \\
-I & 0
\end{matrix} \rp.
\end{align*} 
This proves the second claim. 
\end{proof}
Finally, we also need the following lemma. 
\begin{lemma}\label{diagonallemma}
Suppose that $g \in \SP (2d, \R)$ has the Iwasawa decomposition $g = k a n$ where $k \in {\bf K}$, $a = \mathrm{diag} \lp a_1, a_2, \ldots \rp \in {\bf A}$, and $n \in {\bf N}$. If $\Vert g \textbf{e}_1 \Vert \leqslant \Vert g \textbf{v} \Vert$ for any $\textbf{v} \in \Z^{2d} \setminus \lbrace \0 \rbrace$, then we have $a_1 / a_2 \leqslant 2 / \sqrt{3}$.
\end{lemma}
\begin{proof}
This is an immediate consequence of \cite[Lemma V.1.6]{bekkamayer}: Instead of ensuring that $\left|n_{ij} \right| \leqslant 1/2$ for all $1 \leqslant i < j \leqslant 2d$, we only have to ensure $\left| n_{12} \right| \leqslant 1/2$. This can be done in the symplectic case with the help of \cite[Lemma V.1.5]{bekkamayer} applied to the upper leftmost block $N$ in the block decomposition $n = \lp \begin{smallmatrix} N & M \\ 0 & N^{-\intercal} \end{smallmatrix} \rp$ of $n$.
\end{proof}
We are now ready to prove Proposition \ref{siegeldroid} with $u = m(d)$, where $m(d)$ is defined in Lemma \ref{tooktoolong}.
\begin{proof}[Proof of Proposition \ref{siegeldroid}]
If $d = 1$, the result follows from \cite[Thm. V.1.7]{bekkamayer} since then the symplectic group coincides with the special linear group. We now proceed by induction in $d$. \par 
Assume that $d \geqslant 2$ and that the claim holds for $\SP (2d-2, \R)$. Now, let $g \in \SP(2d, \R)$, and let $\textbf{v}_0 \in \Z^{2d} \setminus \lbrace 0 \rbrace$ be a primitive vector with the property that $g \textbf{v}_0$ is a shortest non-zero vector in the lattice $g \Z^{2d}$. Since $\SP (2d, \Z)$ acts transitively on the set of primitive vectors in $\Z^{2d}$ (see e.g. \cite{mosko}), there is $\gamma \in \SP (2d, \Z)$ such that $\gamma \textbf{e}_1 = \textbf{v}_0$. Then with $g' = g \gamma$, we have
 \begin{align}\label{maybeneededlater}
\Vert g' \textbf{e}_1 \Vert = \Vert g \gamma \textbf{e}_1 \Vert = \Vert g \textbf{v}_0 \Vert \leqslant \Vert g \textbf{v} \Vert,
\end{align} 
for any non-zero $\textbf{v} \in \Z^{2d}$. \par 
We claim that, in order to finish the proof, it is enough to find $\gamma' \in \Gamma$ such that $a' n' \gamma' \in {\bf K} {\bf A}_t {\bf N}$ where $k' a' n' = g'$ is the Iwasawa decomposition of $g'$. Indeed, according to Lemma \ref{tooktoolong}, we will then have
 \begin{align*}
a' n' \in {\bf K} {\bf A}_t {\bf N}_{m(d)} {\bf N}_{\Z} \lp \gamma ' \rp^{-1} \subset {\bf K} {\bf A}_t {\bf N}_{m(d)} \SP (2d, \Z),
\end{align*} 
and hence $g = g' \gamma^{-1} \in {\bf K} {\bf A}_t {\bf N}_{m(d)} \SP (2d, \Z)$.\par 
If $a' = \text{diag}(d_1, d_2, \ldots)$, then we can write
 \begin{align*}
h = a' n' = \lp 
\begin{matrix}
\mathcal{M}\lp d_1, \textbf{x}_1, {\bf 0}, E_0 \rp & \mathcal{M}\lp \lambda, \textbf{y}_1, \textbf{y}_2,F_0 \rp \\
0 & \mathcal{M} \lp d_1^{-1}, {\bf 0}, \textbf{x}_2, E_0^{-\intercal} \rp
\end{matrix} \rp
\end{align*} 
where the symplecticity of $h$ forces the relations
 \begin{align}
\textbf{x}_2 &= -d_1^{-1}E_0^{-\intercal} \textbf{x}_1, \label{thing1}\\
\textbf{y}_2 &= d_1^{-1} \lp E_0 \textbf{y}_1 - F_0 \textbf{x}_1 \rp \label{thing2}.
\end{align} 
It follows from Lemma \ref{symplecticlemma} that 
 \begin{align*}
\lp \begin{matrix}
E_0 &F_0 \\
0 & E_0^{-\intercal} 
\end{matrix} \rp \in \SP (2d-2, \R ).
\end{align*} 
By the inductive hypothesis, we can therefore find 
 \begin{align*}
\gamma '' = \lp \begin{matrix}
\Gamma_1 & \Gamma_2 \\
\Gamma_3 & \Gamma_4
\end{matrix} \rp \in \SP (2d-2, \Z)
\end{align*} 
such that
 \begin{align*}
\lp \begin{matrix}
E_0 &F_0 \\
0 & E_0^{-\intercal} 
\end{matrix} \rp \gamma '' \in {\bf K} {\bf A}_t {\bf N}_{m(d-1)}.
\end{align*} 
We then augment $\gamma ''$ to a $2d \times 2d$ integer matrix $\gamma' $ by letting
 \begin{align*}
\gamma ' := \lp \begin{matrix}
\mathcal{M} (1, {\bf 0}, {\bf 0}, \Gamma_1 ) &\mathcal{M} (0, {\bf 0}, {\bf 0}, \Gamma_2 ) \\
\mathcal{M} (0, {\bf 0}, {\bf 0}, \Gamma_3 ) &\mathcal{M} (1, {\bf 0}, {\bf 0}, \Gamma_4 )
\end{matrix} \rp.
\end{align*}  
By Lemma \ref{symplecticlemma}, $\gamma '$ is symplectic. A computation now shows that
 \begin{align}\label{hellisheavencomparedtothis}
h \gamma' = \lp \begin{matrix}
\mathcal{M} \Bigl( d_1, \Gamma_1^\intercal \textbf{x}_1 + \Gamma_3^\intercal \textbf{y}_1, {\bf 0}, E_0 \Gamma_1 + F_0 \Gamma_3 \Bigr) & \mathcal{M} \Bigl( \lambda, \Gamma_2^\intercal \textbf{x}_1 + \Gamma_4^\intercal \textbf{y}_1, \textbf{y}_2, E_0 \Gamma_2 + F_0 \Gamma_4 \Bigr) \\
\mathcal{M} \lp 0, {\bf 0}, {\bf 0}, E_0^{-\intercal} \Gamma_3 \rp & \mathcal{M} \lp d_1^{-1}, {\bf 0}, \textbf{x}_2, E_0^{-\intercal} \Gamma_4 \rp
\end{matrix}
\rp.
\end{align} 
\par Let us suppose that $\lp \begin{smallmatrix}
E_0 &F_0 \\
0 & E_0^{-\intercal} 
\end{smallmatrix} \rp \gamma ''$ has the Iwasawa decomposition $k'' a'' n''$ with
 \begin{align*}
k'' = \lp \begin{matrix}
K_1 & K_2 \\
K_3 & K_4
\end{matrix} \rp, \quad a'' = \lp  \begin{matrix}
A & 0\\
0 & A^{-1} 
\end{matrix} \rp, \quad n'' = \lp \begin{matrix}
N_1 &M_1 \\
0 & N_1^{-\intercal}
\end{matrix}  \rp .
\end{align*} 
We will now augment $k''$, $a''$, and $n''$ to symplectic $(2d) \times (2d)$ matrices $\tilde{k}$, $\tilde{a}$, and $\tilde{n}$ so that $\tilde{k} \tilde{a} \tilde{n} = h \gamma '$ is the Iwasawa decomposition of $h \gamma'$. To this end, we let
 \begin{align*}
\tilde{k} &:= \lp \begin{matrix}
\mathcal{M} (1, {\bf 0}, {\bf 0}, K_1 ) & \mathcal{M}(0, {\bf 0}, {\bf 0}, K_2) \\
\mathcal{M} (0, {\bf 0}, {\bf 0}, K_3 ) & \mathcal{M} (1, {\bf 0}, {\bf 0}, K_4 )
\end{matrix}
\rp, \\[+1em]
\tilde{a} &:= \lp
\begin{matrix}
\mathcal{M} (d_1, {\bf 0}, {\bf 0}, A ) & 0 \\
0 & \mathcal{M} \lp d_1^{-1}, {\bf 0}, {\bf 0}, A^{-1} \rp
\end{matrix}
\rp,
\end{align*} 
and, for some suitable $\textbf{v}_1, \textbf{w}_1 \in \R^{d-1}$,
 \begin{align*}
\tilde{n} :=  \lp 
\begin{matrix}
\mathcal{M} \lp 1, \textbf{v}_1, {\bf 0}, N_1 \rp & \mathcal{M} \lp d_1^{-1} \lambda, \textbf{w}_1, N_1 \textbf{w}_1 - M_1 \textbf{v}_1, M_1 \rp \\
0 & \mathcal{M} \lp 1, {\bf 0}, -N_1^{-\intercal} \textbf{v}_1, N_1^{-\intercal} \rp
\end{matrix}
\rp.
\end{align*} 
Here it is clear that $\tilde{a}$ is symplectic, and from Lemma \ref{symplecticlemma} we see that also $\tilde{k}$ is. To see that even $\tilde{n}$ is symplectic, we need to verify that the lower rightmost block in $\tilde{n}$ is the inverse-transpose of its upper leftmost block, which is clear; and moreover that we have the identity
 \begin{align}\label{morestufftocheck}
\nonumber &\mathcal{M} \lp 1, \textbf{v}_1, {\bf 0}, N_1 \rp \mathcal{M} \lp d_1^{-1} \lambda, \textbf{w}_1, N_1 \textbf{w}_1 - M_1 \textbf{v}_1, M_1 \rp^\intercal 
\\ &\quad \quad = \mathcal{M} \lp d_1^{-1} \lambda, \textbf{w}_1, N_1 \textbf{w}_1 - M_1 \textbf{v}_1, M_1 \rp \mathcal{M} \lp 1, \textbf{v}_1, {\bf 0}, N_1 \rp^\intercal.
\end{align} 
This follows immediately from the fact that $N_1 M_1^\intercal = M_1 N_1^\intercal$. Using the fact that $\smash{\lp \begin{smallmatrix}
E_0 &F_0 \\
0 & E_0^{-\intercal} 
\end{smallmatrix} \rp \gamma ''} = k'' a'' n''$, we then obtain that 
 \begin{align*}
\tilde{k} \tilde{a} \tilde{n} = 
\lp 
\begin{matrix}
\mathcal{M} \lp d_1, d_1 \textbf{v}_1, {\bf 0}, E_0 \Gamma_1 + F_0 \Gamma_3 \rp & \mathcal{M} \lp \lambda, d_1 \textbf{w}_1, \textbf{u}_1 , E_0 \Gamma_2 + F_0 \Gamma_4 \rp \\
\mathcal{M} \lp 0, {\bf 0}, {\bf 0}, E_0^{-\intercal} \Gamma_3 \rp & \mathcal{M} \lp d_1^{-1}, {\bf 0}, \textbf{u}_2, E_0^{-\intercal} \Gamma_4 \rp
\end{matrix}
\rp
\end{align*} 
with
 \begin{align}
\textbf{u}_1 &= - (E_0 \Gamma_2 + F_0 \Gamma_4 ) \textbf{v}_1 + (E_0 \Gamma_1 + F_0 \Gamma_3) \textbf{w}_1, \label{homealone1}\\
\textbf{u}_2 &= -E_0^{-\intercal}\Gamma_4 \textbf{v}_1 + E_0^{-\intercal} \Gamma_3 \textbf{w}_1. \label{homealone2}
\end{align} 
By (\ref{thing1}), (\ref{thing2}), and (\ref{hellisheavencomparedtothis}) the matrix $\tilde{k} \tilde{a} \tilde{n}$ equals $h \gamma'$ precisely if we have the relations
 \begin{align}
\textbf{v}_1 &= d_1^{-1} \lp \Gamma_1^\intercal \textbf{x}_1 + \Gamma_3^\intercal \textbf{y}_1 \rp, \label{psydub01} \\
\textbf{w}_1 &= d_1^{-1} \lp \Gamma_2^\intercal \textbf{x}_1 + \Gamma_4^\intercal \textbf{y}_1 \rp, \label{psydub02} \\
\textbf{u}_1 &= d_1^{-1} \lp - F_0 \textbf{x}_1 + E_0 \textbf{y}_1  \rp, \label{psydub03} \\
\textbf{u}_2 &= -d_1^{-1}E_0^{-\intercal} \textbf{x}_1. \label{psydub04}
\end{align} 
Since we are free to choose $\textbf{v}_1$ and $\textbf{w}_1$, we simply take (\ref{psydub01}) and (\ref{psydub02}) as definitions. It then follows from (\ref{homealone1}) and (\ref{homealone2}) that (\ref{psydub03}) and (\ref{psydub04}) are satisfied if
 \begin{align*}
\Gamma_4 \Gamma_1^\intercal - \Gamma_3 \Gamma_2^\intercal = I, \quad \quad \Gamma_1 \Gamma_2^\intercal = \Gamma_2 \Gamma_1^\intercal, \quad \quad \Gamma_3 \Gamma_4^\intercal = \Gamma_4 \Gamma_3^\intercal .
\end{align*} 
These relations follow immediately from the fact that $\lp \gamma '' \rp^\intercal$ is symplectic because $\gamma''$ is, and we conclude that with the above choices of $\textbf{v}_1$ and $\textbf{w}_1$, $\tilde{k} \tilde{a} \tilde{n}$ is, indeed, the Iwasawa decomposition of $h \gamma '$.
\par It follows from the induction hypothesis that the entries $a_1, a_2, \ldots, a_{d-1}$ of the diagonal matrix $A$ satisfy the inequality $a_i / a_{i+1} \leqslant 2 / \sqrt{3}$ for $i = 1, \ldots, d-2$. Hence, in order to conclude the proof we must show that $d_1 / a_1 \leqslant 2 / \sqrt{3}$. This can easily be accomplished by appealing to Lemma \ref{diagonallemma}: The matrix $\gamma '$ fixes $\textbf{e}_1$, and therefore, for any $\textbf{v} \in \Z^{2d} \setminus \lbrace 0 \rbrace$,
 \begin{align*}
\lev h \gamma ' \textbf{e}_1 \rev = \lev g' \textbf{e}_1 \rev \leqslant \lev g' \textbf{v} \rev = \lev h \textbf{v} \rev
\end{align*} 
due to (\ref{maybeneededlater}). Replacing $\textbf{v}$ by $\gamma ' \textbf{v} \in \Z^{2d} \setminus \lbrace 0 \rbrace$, we obtain that 
 \begin{align*}
\lev h \gamma ' \textbf{e}_1 \rev \leqslant \lev h \gamma ' \textbf{v} \rev
\end{align*} 
for any non-zero $\textbf{v} \in \Z^{2d}$. By Lemma \ref{diagonallemma}, we therefore have $d_1 / a_1 \leqslant 2 / \sqrt{3}$, and so the proof is concluded.
\end{proof}
\end{appendices}
\vspace{1cm}
\textsc{Acknowledgements.} The author wishes to thank Michael Björklund and Anders Södergren for suggesting the problem and for many illuminating discussions. Moreover, the author is very grateful to the two anonymous referees for their detailed comments and suggestions to a previous version of this article.

\hrulefill 
$ $ \vspace{0.5cm}
\begin{adjustwidth}{}{}
\begin{small}
\textit{Email address:} {\tt holm@math.uni-kiel.de}
\end{small}
\end{adjustwidth}

\end{document}